\documentclass[a4paper,12pt]{amsart}   
\usepackage{amsmath,amssymb,amsxtra}
\usepackage{amsthm}
\usepackage[all]{xy}

\newcommand{\R}{{\mathrm R}}

\newcommand{\Z}{\mathbb{Z}}

\newcommand{\pcom}{{}_{p}^{\wedge}}

\DeclareMathAlphabet\EuR{U}{eur}{m}{n}
\SetMathAlphabet\EuR{bold}{U}{eur}{b}{n}

\newcommand{\Res}{\operatorname{Res}\nolimits}

\newcommand{\Inj}{\operatorname{Inj}\nolimits}

\newcommand{\defeq}{\overset{\text{\textup{def}}}{=}}

\renewcommand{\:}{\colon}

\newcommand{\calb}{\mathcal{B}}
\newcommand{\calc}{\mathcal{C}}
\newcommand{\cald}{\mathcal{D}}
\newcommand{\cale}{\mathcal{E}}
\newcommand{\calf}{\mathcal{F}}

\newcommand{\call}{\mathcal{L}}

\newcommand{\orb}{\mathcal{O}}

\newcommand{\SFL}[1][]{(S#1,\calf#1,\call#1)}
\newcommand{\callq}{\call^q}

\renewcommand{\mod}{\mbox{-}\curs{mod}}
\newcommand{\alg}{\mbox{-}\curs{alg}}

\newcommand{\cat}{\curs{Cat}}

\newcommand{\widebar}[1]{\overset{\mskip3mu\hrulefill\mskip3mu}{#1}
                \vphantom{#1}}

\newcommand{\Id}{\operatorname{Id}\nolimits}

\newcommand{\Inn}{\operatorname{Inn}\nolimits}

\let\oldcirc=\circ
\renewcommand{\circ}{\mathchoice
    {\mathbin{\scriptstyle\oldcirc}}{\mathbin{\scriptstyle\oldcirc}}
    {\mathbin{\scriptscriptstyle\oldcirc}}
    {\mathbin{\scriptscriptstyle\oldcirc}}}

\newcommand{\hclim}[1]{\setbox1=\hbox{\rm hocolim}
    \setbox2=\hbox to \wd1{\rightarrowfill} \ht2=0pt \dp2=-1pt
    \mathop{\vtop{\baselineskip=5pt\box1\box2}}
    _{#1}}

\newcommand{\higherlim}[2]{\displaystyle\setbox1=\hbox{\rm lim}
    \setbox2=\hbox to \wd1{\leftarrowfill} \ht2=0pt \dp2=-1pt
    \setbox3=\hbox{$\scriptstyle{#1}$}
    \ifdim\wd1<\wd3
    \mathop{\hphantom{^{#2}}\vtop{\baselineskip=5pt\box1\box2}^{#2}}_{#1}
    \else
    \mathop{\vtop{\baselineskip=5pt\box1\box2}}\limits_{#1}\nolimits^{#2}
    \fi}

\renewcommand{\hom}{\operatorname{Hom}\nolimits}
\newcommand{\mor}{\operatorname{Mor}\nolimits}

\newcommand{\aut}{\operatorname{Aut}\nolimits}

\newcommand{\Ob}{\operatorname{Ob}\nolimits}

\newcommand{\xxto}[1]{\mathrel{\mathop{%
  \setbox0\hbox{$\ {\scriptstyle#1}\ $}%
  \hbox to \wd0{\rightarrowfill}}^{#1}}%
}

\newcommand{\xto}[2][]{%
  \mathrel{\mathop{%
    \setbox0\vbox{
      \hbox{$\scriptstyle\;\;{#1}\;\;$}%
      \hbox{$\scriptstyle\;\;{#2}\;\;$}%
    }%
    \hbox to\wd0{\rightarrowfill}\displaystyle}%
  \limits^{#2}\ifx{#1}{}\else{_{#1}}\fi}%
}

\newcommand{\longleft}[1]{\;{\leftarrow%
\count255=0 \loop \mathrel{\mkern-6mu}%
    \relbar\advance\count255 by1\ifnum\count255<#1\repeat}\;}
\newcommand{\longright}[1]{\;{\count255=0 \loop \relbar\mathrel{\mkern-6mu}%
    \advance\count255 by1\ifnum\count255<#1\repeat\rightarrow}\;}
\newcommand{\Right}[2]{\overset{#2}{\longright#1}}
\newcommand{\RIGHT}[3]{\mathrel{\mathop{\kern0pt\longright#1}
        \limits^{#2}_{#3}}}

\newcommand{\LEFT}[3]{\mathrel{\mathop{\kern0pt\longleft#1}\limits^{#2}_{#3}}
}
\newcommand{\dRIGHT}[3]{\mathrel{%
   \mathop{\vcenter{\baselineskip=0pt\hbox{$\kern0pt\longright#1$}%
   \hbox{$\kern0pt\longright#1$}}}\limits^{#2}_{#3}}}
\newcommand{\LRIGHT}[3]{\mathrel{%
   \mathop{\vcenter{\baselineskip=0pt\hbox{$\kern0pt\longleft#1$}%
   \hbox{$\kern0pt\longright#1$}}}\limits^{#2}_{#3}}}
\newcommand{\RLEFT}[3]{\mathrel{%
   \mathop{\vcenter{\baselineskip=0pt\hbox{$\kern0pt\longright#1$}%
   \hbox{$\kern0pt\longleft#1$}}}\limits^{#2}_{#3}}}
\newcommand{\onto}[1]{\;{\count255=0 \loop \relbar\joinrel
    \advance\count255 by1
    \ifnum\count255<#1 \repeat \twoheadrightarrow}\;}

\newcommand{\calfq}{\calf^q}

\newcommand{\homf}{\hom_{\calf}}
\newcommand{\homs}{\hom_{S}}

\newcommand{\sylp}{\mathrm{Syl}_p}
\newcommand{\FF}{\mathbb{F}}
\newcommand{\calg}{\mathcal{G}}

\newcommand{\lf}[2]{\def\test{#2}\def\tst{}\ifx\test\tst\calf_{1}%
        \else\calf_{#2}\fi}      

\renewcommand{\ll}[2]{\def\test{#2}\def\tst{}\ifx\test\tst\call_{1}%
        \else\call_{#2}\fi}      
\newcommand{\llx}[2]{\ll{#1}{#2}^{\bullet}}

\newcommand{\Obj}{\mathrm{Obj}}

\newcommand{\da}{\downarrow}

\newcommand{\colimnoarrow}{\operatornamewithlimits{colim}}
\newcommand{\highercolimnoarrow}[2]{\operatornamewithlimits{colim^{#1}}_{#2}}
\newcommand{\higherlimnoarrow}[2]{\operatornamewithlimits{lim^{#1}}_{#2}}

\newcommand{\PreTr}{\operatorname{Pre-Tr}}

\DeclareMathAlphabet\EuR{U}{eur}{m}{n}
\SetMathAlphabet\EuR{bold}{U}{eur}{b}{n}
\newcommand{\curs}{\EuR}
\newcommand{\Ab}{\curs{Ab}}
\newcommand{\Fp}{\ensuremath{\mathbb{F}_p}}

\newcommand{\Sh}{\ensuremath{Sh}}

\newcommand{\Tr}{\ensuremath{\mathrm{Tr}}}
\newcommand{\overcat}[2] {\mbox{\ensuremath{#1\downarrow #2}}}

\theoremstyle{theorem}
\newtheorem{theorem}{Theorem}[section]
\newtheorem{proposition}[theorem]{Proposition}
\newtheorem{lemma}[theorem]{Lemma}
\newtheorem{corollary}[theorem]{Corollary}

\theoremstyle{definition}
\newtheorem{definition}[theorem]{Definition}

\theoremstyle{remark}
\newtheorem{remark}[theorem]{Remark}
\newtheorem{Notation}[theorem]{Notation}

\newcommand{\ack}{\textbf{Acknowledgements: }}

\begin{document}

\title{$p$-local finite group cohomology}

\author{Ran Levi}
\address{Institute of Mathematics, 
         University of Aberdeen,
         Fraser Noble Building 138, 
         Aberdeen AB24 3UE, 
         U.K.}
\email{r.levi@abdn.ac.uk}

\author{K\'ari Ragnarsson}
\address{Department of Mathematical Sciences,
         Depaul University, 
         2320 N.~Kenmore Avenue, 
         Chicago, IL 60614, 
         USA}
\email{kragnars@math.depaul.edu}

\thanks{This project was supported by EPSRC grant GR/S94667/01.}

\subjclass[2010]{20J06, 20D20, 55R40.}

\keywords{fusion system, group cohomology, $p$-local finite group.}

\begin{abstract}
We study cohomology for $p$-local finite groups with non-constant
coefficient systems. In particular we show that under certain
restrictions there exists a cohomology transfer map in this context,
and deduce the standard consequences.
\end{abstract}


\maketitle

\section{Introduction}

A $p$-local finite group is an algebraic object designed to
encapsulate the information modeled on the $p$-completed classifying
space of a finite group. More specifically, it is a triple $\SFL$,
where $S$ is a finite $p$-group, and $\calf$ is a certain category
whose objects are the subgroups of $S$ and whose morphisms are
certain homomorphisms between them, satisfying a list of axioms,
which entitles it to be called a "saturated fusion system over $S$".
The category $\calf$ models conjugacy relations between subgroups of
$S$, while $\call$ is again a category, whose objects are a certain
subcollection of subgroups of $S$, but whose morphism set contain
enough structure as to allow one to associate a classifying space
with $\SFL$ from which the full structure briefly described above
can be retrieved. The classifying space of $\SFL$ is simply the
$p$-completed nerve $|\call|\pcom$.

The theory of $p$-local finite groups was introduced in \cite{BLO2}
and has been studied quite extensively by various authors. In
particular the mod $p$ cohomology of the classifying space of a
$p$-local finite group satisfies a "stable elements theorem",
identical in essence to the corresponding statement for the
cohomology of finite groups with constant mod $p$ coefficients
\cite[Thm B]{BLO2}. However, as we shall see below, the analogy
does not carry forward when one considers a more general setup,
i.e., the cohomology of $p$-local finite groups in the context of
functor cohomology. In particular we will present counter examples
to the statement that the cohomology of the $p$-local group, in this
sense, is given as the stable elements in the cohomology of its
Sylow subgroup with restricted coefficient.

The purpose of this paper is to study the cohomology of $p$-local
finite groups with arbitrary coefficients, and in particular to
establish an algebraic definition of a transfer map for $p$-local
finite groups in this setting. We deal with the cases where a
geometric transfer must exist, that is, when there is a finite index
covering space associated with the given setup. The necessary theory
is provided by the papers \cite{BCGLO1, BCGLO2}, where "subgroups"
of a $p$-local finite group of a given "index" are defined and
studied. In particular it is shown in \cite{BCGLO2} that with a
given $p$-local finite group one can associate two finite groups --
a $p$-group and a $p'$-group -- and that, up to equivalence, $p$-local
subgroups of $p$-power index and $p'$ index are in one to one
correspondence with subgroups of these two groups. Subgroups of
indices which are neither $p$-power nor $p'$ can also be studied
under restricted conditions, which essentially follow from the two
basic cases. In the current article we show the existence of a
"transfer map" to the cohomology of a $p$-local finite group from
the cohomology of any subgroup of index prime to $p$ with respect to
any system of coefficients, and from  the cohomology of any subgroup
of $p$-power index with respect to locally constant systems of
coefficients. These transfer maps carry similar properties to the
standard group cohomology transfer, and can be used to study
cohomology of $p$-local finite groups with non-constant
coefficients.

We now explain our setup and state our results. To a finite group $G$, 
one can associate a category $\calb(G)$ with a single object whose 
endomorphism monoid is $G$. A $G$-module $M$ can be regarded as a 
functor on $\calb(G)$, and the cohomology of $G$ with coefficients in $M$ can be computed as the higher limits of that functor. By analogy, a
coefficient system on a $p$-local finite group $\SFL$ is a functor
$M\colon \call\to \Ab$, where $\Ab$ is the category of abelian
groups (or sometimes $\Z_{(p)}$-modules, in which case we say that
the system is $p$-local), and the cohomology groups $H^*(\call,M)$ 
are defined as the higher limits of $M$ over $\call$. Unless otherwise specified, we will work with coefficient systems which are covariant functors. A coefficient system $M$ is \emph{locally constant} if it sends every morphism in $\call$ to an isomorphism of modules.

As already indicated, we only define a transfer for arbitrary
systems of coefficients under certain, rather restrictive,
conditions. We start our discussion with a proposition, due to Bob
Oliver, which shows that in general one cannot hope to do much
better.

\begin{proposition}\label{Bob's-example}
There exist a $p$-local finite group $\SFL$ and a locally constant
$p$-local system of coefficients $M$ on $\SFL$ such that
$\widetilde{H}^*(\call, M)\neq 0$  but $\widetilde{H}^*(S, \iota^*
M)=0$, where $\iota^*M$ means the restriction of $M$ to $S$.
\end{proposition}

Thus the proposition shows that unlike the situation in ordinary
group cohomology, the cohomology of a $p$-local finite group with
$p$-local coefficients  is not always a retract of the cohomology of
its Sylow subgroup with the restricted coefficients.

If $\SFL$ is a $p$-local finite group, and $(S_0,\calf_0,\call_0)$
is a subgroup of $p$-power index or index prime to $p$, then one has
a an associated covering space up to homotopy $\iota\colon
|\call_0|\to |\call|$, and thus if $M$ is a
$\Z_{(p)}[\pi_1(|\call|)]$-module, then there is a transfer map
$H^*(|\call_0|,\iota^* M)\to H^*(|\call|, M)$ associated with the
covering. We present a categorical analog of this setup, and use it
to prove our first theorem.
\begin{theorem}\label{Covers}
Let  $\SFL$ be a $p$-local finite group, and let
$(S_0,\calf_0,\call_0)$ be a subgroup of $p$-power index or index
prime to $p$. Then for any locally constant coefficient system $M$
on $\call$, there exists a map $\Tr\colon H^*(\call_0, M_0)\rightarrow
H^*(\call, M)$, where $M_0$ denotes the restriction of $M$ to
$\call_0$, such that the composite
\[H^*(\call, M)\xrightarrow{\Res} H^*(\call_0, M_0)\xrightarrow{\Tr} H^*(\call, M)\]
is multiplication by the index. Furthermore, the map $\Tr$ coincides
with the transfer map associated with the covering $|\call_0|\to|\call|$.
\end{theorem}

Theorem \ref{Covers} is proven below as Theorem
\ref{plf-subgrps-loc-const-coeff}. Next, we specialize  to $p$-local
finite subgroups of index prime to $p$. It is in this setting that
we are able to obtain our most general result.

\begin{theorem}\label{p' index}
Let $\SFL$ be a $p$-local finite group, and let
$(S_0,\calf_0,\call_0)$ be a subgroup of index prime to $p$. Let $M$
be an arbitrary system of coefficients on $\call$. Then
there exists a map
\[\Tr\colon H^*(\call_0, M_0)\to H^*(\call, M),\]
where $M_0$ denotes the restriction of $M$ to $\call_0$, such that
the composite $\Tr\circ\Res$ is given by multiplication by the
index. Furthermore, the map $\Tr$ satisfies a double coset formula
and Frobenius reciprocity.
\end{theorem}

The construction of the transfer map is carried out in Section 4.
The precise statements and proof of its properties appear below as
Propositions \ref{prop:Norm}, \ref{double-cosets} and
\ref{prop:Frob}.

As a final standard application we conclude a stable elements
theorem in our context. The statement of the theorem however
requires some extra preparation, and will therefore not be stated
here. The reader who is familiar with \cite{BCGLO2} is referred to
Theorem \ref{Stable-Elts} for the precise statement.

The paper is organized as follows. In Section 1 we recall the
definition of a $p$-local finite group, and revise some of the
background material necessary for our discussion. Section 2 is
devoted to some basic concepts of homological algebra of functors,
including  a discussion of cohomology with locally constant
coefficients. We start Section 3 by presenting a family of examples
which show that in general one cannot expect a stable elements
theorem to hold for $p$-local finite groups even with respect to
locally constant coefficients. This is followed by a discussion of
the transfer for locally constant coefficient systems and subgroups
of $p$-power index and index prime to $p$. Still in Section 3 we set
up the background for the construction of a transfer map for
subgroup of index prime to $p$. The construction itself is carried
out in Section 4, where we also present a cochain level construction
of the map, and a geometric interpretation of the construction for
locally constant coefficients, where a transfer map exists for
geometric reasons.  Finally in Section 5 we derive the
standard consequences of the existence of a transfer map to the
cohomology theory of $p$-local finite groups with nontrivial
coefficients.

\ack The authors would like to thank Bob Oliver for a number of useful
conversations while we were working on this project, and
particularly for his illuminating counter example (Section 3). The
second named author would like thank Lukas Vokrinek for giving him a
copy of his bachelor thesis, which contains useful background
material on functor cohomology.


\section{$p$-local finite groups}
In this section we recall some of the basic concepts of $p$-local finite group theory. We also recall the \emph{Stable Elements Theorem}  for the cohomology of $p$-local finite groups with $p$-local constant coefficients. The reader is referred to \cite{BLO2, BCGLO1, BCGLO2} for more detail.

\subsection{Saturated Fusion Systems}
The fundamental object
underlying a $p$-local finite group is a finite $p$-group and a
fusion system over it. The concept is originally due to Puig \cite{Puig2}, but we will use the simpler, equivalent definition from \cite{BLO2}.

For a group $G$ and subgroups $P,Q\le G$ we denote by $\hom_G(P,Q)\subseteq\hom(P,Q)$
the set of all homomorphisms $P \to Q$ obtained by restriction
of inner automorphisms of $G$ to $P$.  The set of all elements $g\in G$ such that $gPg^{-1}\le Q$ is called the transporter set in $G$ from $P$ to $Q$, and is denoted $N_G(P,Q)$. Thus $\hom_G(P,Q) =N_G(P,Q)/C_G(P)$. If $g\in G$, then we denote by $c_g$ the conjugation $x \to gxg^{-1}$.

\begin{definition}[{\cite{Puig2} and \cite[Definition~1.1]{BLO2}}]
\label{fus.sys.} A \emph{fusion system} over a finite $p$-group $S$ is a
category $\calf$, where $\Ob(\calf)$ is the set of all subgroups of
$S$, and which satisfies the following two properties for all
$P,Q\le{}S$:
\begin{enumerate}
\item  $\homs(P,Q) \subseteq \homf{(P,Q)} \subseteq \Inj(P,Q)$; and
\item  each $\varphi\in\homf(P,Q)$ is the composite of an isomorphism in
$\calf$ followed by an inclusion.
\end{enumerate}
\end{definition}

If $\calf$ is a fusion system over a finite $p$-group $S$, then two
subgroups $P,Q\le S$ are said to be $\calf$-conjugate if they are
isomorphic as objects in $\calf$.

\newcommand{\out}{\mathrm{Out}}

\begin{definition}[{\cite{Puig2}, and \cite[Definition~1.2]{BLO2}}]
\label{sat.Frob.} Let $\calf$ be a fusion system over a $p$-group
$S$. A subgroup $P\le{}S$ is said to be
\begin{enumerate}
\item  \emph{fully centralized in $\calf$} if
$|C_S(P)|\ge|C_S(P')|$ for all $P'\le S$ which are $\calf$-conjugate to $P$;
\item  \emph{fully normalized in $\calf$} if
$|N_S(P)|\ge|N_S(P')|$ for for all $P'\le S$ which are $\calf$-conjugate to $P$;
\item $\calf$-centric if $C_S(P') = Z(P')$ for all $P'\le S$ which are $\calf$-conjugate to $P$.
\item $\calf$-radical if $\out_{\calf}(P)\defeq \aut_\calf(P)/\Inn(P)$ does not contain a nontrivial normal $p$-subgroup.
\end{enumerate}
A fusion system $\calf$ is said to be  \emph{saturated} if the
following two conditions hold:
\begin{enumerate}
\item[(I)] For every $P\le{}S$ that is fully normalized in $\calf$, $P$ is fully
centralized in $\calf$ and $\aut_S(P)\in\sylp(\aut_\calf(P))$.
\item[(II)] If $P\le{}S$ and $\varphi\in\hom_\calf(P,S)$ are such that $P'=\varphi(P)$ is
fully centralized, then $\varphi$ extends to a morphism $\widebar{\varphi}\in\hom_\calf(N_\varphi,S)$, where
    $$ N_\varphi = \{ g\in{}N_S(P) \,|\, \varphi c_g\varphi^{-1} \in
    \aut_S(P') \}. $$
\end{enumerate}
\end{definition}

If $G$ is a finite group and $S\in\sylp(G)$, then the category
$\calf=\calf_S(G)$, whose objects are all subgroups $P\le S$, and
whose morphisms are $\homf(P,Q)=\hom_G(P,Q)$, is a saturated fusion
system (\cite[Proposition~1.3]{BLO2}).

\subsection{Centric Linking Systems}
Let $\calf^c\subseteq\calf$ denote the full
subcategory whose objects are the $\calf$-centric subgroups of $S$.

\begin{definition}[{\cite[Definition~1.7]{BLO2}}]  \label{L-cat}
Let $\calf$ be a fusion system over the $p$-group $S$.  A
\emph{centric linking system associated to $\calf$} is a category
$\call$ whose objects are the $\calf$-centric subgroups of $S$,
together with a functor $\pi \:\call\to\calf^c$, and
``distinguished'' monomorphisms $P\Right1{\delta_P}\aut_{\call}(P)$
for each $\calf$-centric subgroup $P\le{}S$, which satisfy the
following conditions.
\begin{enumerate}
\renewcommand{\labelenumi}{\textup{(\Alph{enumi})}}
\item  $\pi$ is the identity on objects and surjective on morphism sets.  For each pair of objects
$P,Q\in\Ob(\call)$, $Z(P)$ acts freely on $\mor_{\call}(P,Q)$ by
composition (upon identifying $Z(P)$ with
$\delta_P(Z(P))\le\aut_{\call}(P)$), and $\pi$ induces a bijection
        $$ \mor_{\call}(P,Q)/Z(P) \Right5{\cong} \homf(P,Q). $$

\item  For each $\calf$-centric subgroup $P\le{}S$ and each $x\in{}P$,
$\pi(\delta_P(x))=c_x\in\aut_{\calf}(P)$.

\item  For each $f\in\mor_{\call}(P,Q)$ and each $x\in{}P$,
$f\circ\delta_P(x)=\delta_Q(\pi(f)(x))\circ{}f$.
\end{enumerate}

A \emph{$p$-local finite group} is a triple $\SFL$,
where $S$ is a finite $p$-group, $\calf$ is a saturated fusion
system over $S$, and $\call$ is a centric linking system associated
to $\calf$.  The \emph{classifying space} of a $p$-local finite group $\SFL$ is
the $p$-completed nerve $|\call|\pcom$.
\end{definition}

If $\calf=\calf_S(G)$ for some finite group $G$, then $P\le S$ is
$\calf$-centric if and only if $P$ is \emph{$p$-centric} in $G$;
that is, if and only if $Z(P)\in\sylp(C_G(P))$, or equivalently if and
only if $C_G(P) \cong Z(P) \times C_G'(P)$, where $C_G'(P)$ is a
group of order prime to $p$. The centric linking system of $G$ is
defined to be the category $\call^c_S(G)$, whose objects are the
subgroups of $S$ that are $p$-centric in $G$, and whose morphism
sets are $\mor_\call(P,Q) = N_G(P,Q)/C_G'(P)$. The triple
$(S,\calf_S(G), \call_S^c(G))$ then forms a $p$-local finite group whose
classifying space is equivalent to $BG\pcom$ \cite{BLO2}.

\subsection{Compatible Systems of Inclusions} \label{sub:InclRes}
For a $p$-local finite group $\SFL$, a morphism $\varphi \colon P
\to Q$ in $\call$ can be thought of as a "lift" of the group
homomorphism $\widebar{\varphi} = \pi(\varphi)$ in the fusion system
$\calf$, and the number of such lifts is $|Z(P)|$. It will often be
convenient to extend concepts associated to group homomorphisms to
morphisms in $\call$. For instance, we define the image of $\varphi$
to be the image of $\widebar{\varphi}$, and denote it by
$\varphi(P)$ or $P^\varphi$. (Observe that since the morpism is
written on the right, we have $P^{\psi\circ \varphi} =
(P^{\varphi})^\psi$.)  Just as in $\calf$, a morphism in $\call$ can
be restricted to a subgroup, and also induces an isomorphism from
its source  to its image. To make sense of these restrictions, we
first need a good notion of "inclusions" in $\call$. This is
developed in \cite[Definition 1.11]{BCGLO2}, and we recall the
definitions here.

\begin{definition}[{\cite[Def. 1.11(b)]{BCGLO2}}]
Let $\SFL$ be a $p$-local finite group.
A \emph{compatible set of inclusions for $\call$} is a choice of morphisms $\iota_P^Q \in \mor_{\call}(P,Q)$,
one for each pair of $\calf$-centric subgroups $P \leq Q$, such that $\iota_S^S = \Id_S$, and
the following hold for all $P \leq Q \leq R$,
 \begin{enumerate}
 \item[(i)] $\pi(\iota_P^Q)$ is the inclusion $P \hookrightarrow Q$;
 \item[(ii)] $\iota_Q^R \circ \iota_P^Q = \iota_P^R$;
 \end{enumerate}
\end{definition}

We often write $\iota_P$ for $\iota_P^S$. The existence of a compatible set
of inclusions for $\call$ is proved in \cite[Proposition 1.13]{BCGLO2}. A
compatible set of inclusions for $\call$ allows us to talk about restrictions of morphisms in $\call$. That is, for a morphism $\varphi \in \mor_{\call}(P,Q)$ and
$\calf$-centric subgroups $P' \leq P$ and $Q' \leq Q$ such that $\varphi\circ\iota_{P'}^{P}(P') \leq Q'$,
there is a unique morphism $\varphi_{P'}^{Q'} \in \mor_{\call^q}(P',Q')$ with
$ \iota_{Q'}^Q \circ \varphi_{P'}^{Q'} = \varphi \circ \iota_{P'}^{P}.$ We refer to
$\varphi_{P'}^{Q'}$ as the restriction of $\varphi$ to $P'$.  To simplify notation, we will often write $\varphi_{|P'}$ instead of $\varphi_{P'}^{Q'}$
when there is no danger of confusion. If $P'$ is clear from the context, we will sometimes omit it from the notation as well.

Fix a compatible system of inclusions. Let $\pi_1(|\call|,S)$ be the fundamental group of
$|\call|$ with basepoint at the vertex $S$, and let $\calb(\pi_1(|\call|,S))$
be the category associated to the fundamental group $\pi_1(|\call|,S)$.
One obtains a functor
 \[ J \colon \call \to \calb(\pi_1(|\call|,S)) \]
that sends each object to the unique object in the target, and sends a morphism $f \colon P \to Q$ to the class
of the loop $ \iota_Q * f * \iota_P^{-1} $. Composing with the
distinguished monomorphism $\delta_S \colon S  \to \aut_{\call}(S)$, we get a functor
 \[ \calb(j) \colon \calb(S) \xrightarrow{\calb(\delta_S) } \calb(\aut_{\call}(S)) \subseteq \call \xrightarrow{J} \calb(\pi_1(|\call|,S)), \]
and a corresponding homomorphism
 \[ j\colon S \to \pi_1(|\call|). \]
A construction of the functor $J$  in a more general setting is described in Subsection \ref{loc-const-coeff}.

\subsection{A Stable Elements Theorem}
The cohomology of a finite group $G$ with coefficients in a $p$-local
module can be computed by a fundamental result due to Cartan and
Eilenberg \cite{CE} (known as the \emph{stable elements theorem}).
Their original statement can be reinterpreted as follows. Let
$\orb_S(G)$ denote the category whose objects are the subgroups of
$S$, and whose morphisms $P\to Q$ are representations ($Q$-conjugacy
classes of homomorphisms) induced by conjugation in $G$. Given a
$\Z_{(p)}[G]$-module $M$, there is a functor on $\orb_S(G)^{op}$
which sends $P\le S$ to $H^*(P,M)$, where $M$ becomes a
$\Z_{(p)}[P]$-module via restriction. The Cartan-Eilenberg stable elements theorem
can then be restated as claiming that the following isomorphism holds
\[H^*(G,M) \cong \lim_{\orb_S(G)^{op}}H^*(-,M).\]

Theorem 5.8 in \cite{BLO2} is the analogous statement for $p$-local
finite groups, and where the module of coefficients is the field
$\Fp$.  Specifically, if $(S,\calf,\call)$ is a $p$-local finite
group, and $\orb(\calf)$ is the orbit category of $\calf$, i.e., the
category with the same objects, and where morphisms $P\to Q$ are
given by $\hom_\calf(P,Q)/\Inn(Q)$, then
\[H^*(|\call|,\Fp) \cong \lim_{\orb(\calf)^{op}}H^*(-,\Fp).\]
From this one can deduce that the same statement is true for any
$\Z_{(p)}$-module of coefficients with a trivial $\pi_1(|\call|)$-action.

More generally, using a stable transfer construction it is shown in
\cite{KR:ClSpec} (see also \cite{CM}) that the stable elements theorem holds for any
(non-equivariant) stably representable cohomology theory.
However, as we will observe at the end of  Section \ref{homological algebra}, this statement is far
from being true for cohomology with nontrivial coefficients.

\subsection{Finite Index Subgroups in $p$-local Finite Groups }

We now discuss the setup in which we are able to define a transfer
map on $p$-local group homology and cohomology.
Given a saturated fusion system $\calf$ over a $p$-group $S$, the
paper \cite{BCGLO2} defines what it means to be a subsystem of
$\calf$ of a finite index, which in this context is either a power
of $p$ or prime to $p$. We refer to the latter case as a $p'$ index
subsystem. The paper also gives a classification of all
subsystems of $\calf$ of $p$-power or $p'$ index, which we recall here. We point out that generically (i.e., without considering iterations) these are the only cases where it makes sense to talk about a subsystem where the notion of an index is well defined. 

To any saturated fusion system $\calf$ over a $p$-group $S$, one
associates two finite groups $\Gamma_p(\calf)$ and
$\Gamma_{p'}(\calf)$. If $\calf$ admits an associated centric
linking system $\call$, then these groups turn out to be the maximal
$p$-power and $p'$ quotients of $\pi_1(|\call|)$, respectively. Both
groups depend only on $\calf$ (and not on the existence or nature of
an associated centric linking system $\call$). The group
$\Gamma_p(\calf)$ is given by $S/O^p_\calf(S)$, where the divisor is
the hyperfocal subgroup of $S$ with respect to $\calf$ (see
\cite[Sec. 2]{BCGLO1}). If $\calf$ admits an associated centric
linking system, then $\Gamma_p(\calf)\cong\pi_1(|\call|\pcom)$. The
group $\Gamma_{p'}(\calf)$ has a more complicated description in
\cite{BCGLO2}, but was observed  by Aschbacher to be
$\pi_1(|\calf^c|)$.

Subsystems of $\calf$ of $p$-power or $p'$ index are
in bijective correspondence with subgroups of $\Gamma_p(\calf)$ and
$\Gamma_{p'}(\calf)$, respectively. If $\SFL$ is a $p$-local finite
group, and we let $\Gamma$ denote either $\Gamma_p(\calf)$ or
$\Gamma_{p'}(\calf)$, then with each subgroup $H\le\Gamma$ one has
an associated $p$-local finite group $(S_H,\calf_H,\call_H)$, and
$|\call_H|$ is homotopy equivalent to a covering space of $|\call|$
with fibre $\Gamma/H$.

It is for this type of subgroups of a $p$-local finite group that we
are able to define a transfer in cohomology (or homology) with
coefficients  $M \in \call\mod$. In the case of a subgroup of $p'$ index,
we can do so in full generality, and we define a cochain-level transfer
in cohomology for any functor of coefficients. These general
methods do not carry over to the $p$-power index case, but we can still
define a transfer map for locally constant coefficient systems.

We next state a summary of the main classification result from
\cite{BCGLO2} for fusion subsystem of $p$-power or $p'$ index. Some
minor modifications of the original statement will be dealt with in
the proof. Before stating the theorem, we set up some notation.

Let $\SFL$ be a $p$-local finite group, and let $\calf^q$ and
$\call^q$ be the associated quasicentric fusion and linking systems
\cite[Sec. 3]{BCGLO1}. Let $\Gamma$ denote $\Gamma_p(\calf) =
\pi_1(|\call|\pcom,S)$ or $\Gamma_{p'}(\calf) = \pi_1(|\calf^c|,S)$,
where in both cases $S$ denotes the basepoint given by the vertex
$S$. Fix a compatible choice of inclusions $\{\iota_P^Q\}$ for
$\callq$, and let $J \colon \call^q\to \calb(\pi_1(|\call^q|, S)$
be the resulting functor as defined in Section \ref{sub:InclRes}.
Composing with the obvious projection
$\widehat{\theta}\colon\pi_1(|\call^q|,S)=\pi_1(|\call|,S)\to
\Gamma$, one gets a functor
\[\widehat{\Theta}\colon\callq\to \calb(\Gamma).\]
We will denote the restriction of $\widehat{\Theta}$ to $\call$ by
the same symbol. Notice that $\widehat{\Theta}$ sends the chosen
inclusions in $\call^q$ to the identity (since $J$ does). Let
$\theta\colon S\to\Gamma$ denote the restriction of
$\widehat\Theta$ to $S$ via the monomorphism $\delta_S \colon S \to
\aut_{\call^q}(S) \subseteq \call^q$. (Equivalently, $\theta =
\widehat{\theta}\circ j$, where $j\colon S \to \pi_1(|\call|,S)$ is
the homomorphism defined in Section \ref{sub:InclRes}.) For any
subgroup $H\le\Gamma$, let $\llx{\widehat{\Theta}}H\subseteq\callq$
be the subcategory with the same objects and with morphism set
$\widehat{\Theta}^{-1}(H)$, and let
$\callq_H\subseteq\llx{\widehat{\Theta}}H$ be the full subcategory
obtained by restricting to subgroups of $S_H\defeq\theta^{-1}(H)$.
Finally, let $\lf{\widehat{\Theta}}H$ be the fusion system over
$S_H$ generated by $\pi(\callq_H)\subseteq\calfq$ and restrictions
of morphisms, and let $\call_H\subseteq\callq_H$ be the full
subcategory on those objects which are $\calf_H$-centric.

\begin{theorem}[\cite{BCGLO2}]\label{Classification-summary}
Let $\SFL$ be a $p$-local finite group. Then, with the notation
above the following are satisfied.
\begin{enumerate}
\item[(i)]$\lf{\Theta}H$ is a saturated fusion system over $S_H$,
and $\call_H$ is an associated centric linking system. Thus the
triple $(S_H,\calf_H,\call_H)$ is a $p$-local finite group.
\item[(ii)] If $\Gamma=\Gamma_{p'}(\calf)$ then a  subgroup $P\le{}S_H$ is
$\lf{\Theta}H$-centric (fully $\calf_H$-centralized, fully
$\calf_H$-normalized) if and only if it is $\calf$-centric (fully
$\calf$-centralized, fully
  $\calf$-normalized). If $\Gamma= \Gamma_p(\calf)$ then the same
  statements hold, replacing centric by quasicentric.
\item[(iii)] There is a 1--1 correspondence between
subgroups $H\le \Gamma$ and $p$-local subgroups of  $\SFL$ with
$p$-power or $p'$ index (as appropriate). The correspondence is given by
$H\longleftrightarrow (S_H,\calf_H,\call_H)$.
\item[(iv)] $|\call_H|$ is homotopy equivalent to the
covering space of $|\callq|\simeq|\call|$ with fibre $\Gamma/H$.
\item[(v)] The homomorphism $\aut_\call(S) \to \Gamma$ induced by
the restriction of $\widehat{\Theta}\colon \callq \to \calb(\Gamma)$
to $\aut_{\call}(S)$ is surjective.
\end{enumerate}
\end{theorem}
\begin{proof}
Parts (i), (ii), (iii) and (iv) are included in Proposition 3.8 and
Theorem 3.9 of \cite{BCGLO2}. Part (v) is clear in the case of $\Gamma_p(\calf) = S/O^p_\calf(S)$, and follows from the definition of $\Gamma_{p'}(\calf)$ given in \cite[Thm 5.4]{BCGLO2}).
\end{proof}


\section{Homological algebra of functors} \label{homological algebra}
In this section we  develop the background we need from homological algebra. Most or all the results we present here are well known to the expert, but are included here for the convenience of the reader.

Throughout this  paper, let
$\R$ denote a fixed commutative ring with a unit.
Let $\R\mod$ denote the category of (left) $\R$-modules,
and let $\R\alg$ denote the category of (left) $\R$-algebras.
All categories we consider in this article will be small.
For a category $\calc$, let $\calc\mod$ and $\calc\alg$ be the
categories whose objects are functors \mbox{$\calc \to \R\mod$}
or \mbox{$\calc \to \R\alg$}, respectively, and whose
morphisms are natural transformations. We will refer
to an object of $\calc\mod$ as a \emph{$\calc$-module}, and
to an object of $\calc\alg$ as  a \emph{$\calc$-algebra}. Depending on context we may sometimes refer to objects of $\calc\mod$ as \emph{system of coefficients on $\calc$},
or a \emph{$\calc$-diagram of $\R$-modules}, and similarly for objects of $\calc\alg$. Notice that we have not discussed variance of functors at all. By convention, all functors we deal with are \emph{covariant}. If we need to discuss contravariant functors (and in certain contexts we will), we shall consider them as objects in the module category over the opposite category.

\subsection{Functor Cohomology.}
Given an
$\R$-module $M$, the constant $\calc$-diagram $M_\calc$  is
the functor which takes every object in $\calc$ to $M$ and every
morphism to the identity. The inverse limit on a category $\calc$ is a functor
\[\lim_{\calc} \colon \calc\mod \to \R\mod.\]
It comes with a morphism $\displaystyle \left(\lim_\calc U\right) \xrightarrow{\lambda_U} U$,
for any $\calc$-diagram $U$ of $\R$-modules, and
is characterized by the universal property that if $M$ is any
$\R$-module, and if $\alpha\colon M_\calc\to U$ is any natural transformation,
then there exists a
unique $\R$-module homomorphism $\widehat{\alpha}\colon M\to
\lim_\calc U$ such that $\alpha = \lambda_U \circ \widehat{\alpha}_\calc$, where $\widehat{\alpha}_\calc$ is the functor induced on the respective constant diagrams by $\widehat{\alpha}$.

The universal property of the inverse limit functor implies an
obvious identification:
\begin{equation} \label{eq:limashom}
   \lim_{\calc}U \cong \hom_{\calc\mod}(\R_\calc,U).
\end{equation}
This shows in particular that the inverse limit functor is left exact.
Its right derived functors applied to $U\in\calc\mod$ are usually
referred to as the \emph{higher limits of $U$}, or as the
\emph{cohomology of $\calc$ with coefficients in $U$} (being the right derived functors of the Hom functor). Thus if
\mbox{$U \to I_\bullet$} is an injective resolution of $U$ in
$\calc\mod$, and \mbox{$P_{\bullet} \to \R_\calc$} is a projective
resolution of $\R_\calc$ in $\calc\mod$, then
\[ \higherlimnoarrow{i}{\calc} U = H^{i}(\hom_{\calc\mod}(\R_\calc,I_\bullet)) =
H^{i}(\hom_{\calc\mod}(P_\bullet,U) ).\]
\begin{Notation}
Throughout this article we use the notation $H^*(\calc, U)$ to
denote $\underset{\calc}{\lim}^* U$. This is standard notation in
the subject, and is better suited for purposes.
\end{Notation}

Dual to the limit functor there is the colimit functor
\[\displaystyle \highercolimnoarrow{}{\calc}\colon\calc\mod\to\R\mod.\] One defines the
\emph{homology of $\calc$ with coefficients in $U\in\calc\mod$} as the \emph{higher colimits of $U$},  i.e. the left derived functors of the colimit functor, applied to $U$. In this section,
and throughout the paper, we focus on cohomology and mostly leave the reader to dualize
the discussion to obtain analogous results in homology.
\subsection{Cup Products in Functor Cohomology.} Let $\calc$ be a small category. For
chain complexes $C_\bullet$ and $C'_\bullet$ in $\calc\mod$, one
obtains a chain complex $(C \otimes C')_\bullet$ with
\[(C\otimes C')_n \defeq \sum_{k+l=n} C_k \otimes C'_l \]
and differential $\delta$ induced by the differentials of
$C_\bullet$ and $C'_\bullet$ (which we also denote by $\delta$) via
the formula \mbox{$\delta (x\otimes y) = \delta x \otimes y + (-1)^k
x \otimes \delta y$} for \mbox{$x \in C_k$}, \mbox{$y \in C'_l$}.

It is a standard result that if both $C_\bullet$ and $C'_\bullet$
are exact then $(C \otimes C')_\bullet$ is exact, and likewise that
$(C \otimes C')_\bullet$ is projective if both $C_\bullet$ and
$C'_\bullet$ are  projective.
In particular, given projective resolutions \mbox{$P_{\bullet} \to
\R_\calc$} and \mbox{$P'_{\bullet} \to \R_\calc$} of the constant
functor on $\calc$, one obtains a projective resolution \mbox{$(P
\otimes P')_{\bullet} \to \R_\calc \otimes \R_\calc \cong \R_\calc$}

Let $M, M'\in\calc\mod$. For \mbox{$\sigma \in
\hom_{\calc\mod}(P_k,M) $} and \mbox{$ \sigma' \in
\hom_{\calc\mod}(P'_l,M')$}, let \mbox{$\sigma \times \sigma' \in
\hom_{\calc\mod}((P \otimes P')_{k+l},M \otimes M')$} be the natural
transformation
\[ \sigma \times \sigma' \colon (P \otimes P')_{k+l} \xrightarrow{proj} P_k
\otimes P'_l \xrightarrow{\sigma \otimes \sigma'} M \otimes M'.\]
This gives a bilinear pairing
\[ \hom_{\calc\mod}(P_\bullet,M) \otimes \hom_{\calc\mod}(P'_\bullet,M') \to
\hom_{\calc\mod}((P \otimes P')_\bullet,M \otimes M')\]
\[ \sigma\otimes\sigma' \longmapsto \sigma\times\sigma'.\]
One can check that this is a map of cochain complexes, so after
taking cohomology we obtain a bilinear pairing on higher limits
\[ H^*(\calc, M) \otimes H^*(\calc, N)
\longrightarrow H^*(\calc, M \otimes N),\]
\[ x \otimes y \mapsto x \times y.\]
One can furthermore check that this is independent of the choice of
projective resolutions. This pairing is called the cross product
pairing.

For any $A\in \calc\alg$ one has a  multiplication transformation \mbox{$\mu \colon A
\otimes A \to A$} which induces a
homomorphism on cohomology. Composing with the cross
product pairing, we obtain the cup product.
\begin{definition}
Let $A\in\calc\alg$. For $x \in H^k(\calc, A)$ and
$y \in H^{l}(\calc, A)$, the cup product $x \cup y \in
H^{k+l}(\calc, A)$, (or $xy$,) is the image of $x
\otimes y$ under the homomorphism
 \[H^{*}(\calc, A) \otimes H^{*}(\calc, A)
 \xrightarrow{\times} H^{*}(\calc, A \otimes A) \xrightarrow{ \mu_* }
 H^{*}(\calc, A).\]
\end{definition}

The cup product constructed here has the algebraic properties we
expect, as stated in the next proposition. The proof is routine.
\begin{proposition}\label{cup properties}
For a functor $A\in\calc\alg$, the cup product on $H^*(\calc, A)$ is
associative, graded-commutative, and has a multiplicative unit. Thus
$H^*(\calc, A)$ is a graded-commutative ring with a
unit.
\end{proposition}

If \mbox{$F \colon \calc\to \cald$} is any functor, one has an exact functor
\mbox{$F^* \colon \cald\mod\to \calc\mod $} defined by
$F^*(\alpha) \defeq \alpha\circ F$.
The next proposition says that the cup product is natural in both
$\calc$ and $A$. Again, the proof is routine.

\begin{proposition} Let $F \colon \calc \to \cald$ be a functor,
let $A, B\in \cald\alg$, and let $\eta \colon A \to B$
be a natural transformation. Then the following hold for $x,y \in
H^*(\cald,A)$:
\begin{enumerate}
\item[(a)] The  homomorphism
  \mbox{$F^* \colon H^{*}(\cald, A) \to
  H^{*}(\calc, F^*A) $} induced by $F$
satisfies
\[ F^*(x \cup y) = F^*x \cup F^*y.\]

\item[(b)] The  homomorphism
  \mbox{$ \eta_* \colon H^{*}(\cald, A) \to
  H^{*}(\cald, B) $} induced by $\eta$
satisfies
\[ \eta_*(x \cup y) = \eta_*x \cup \eta_*y. \]
\end{enumerate}
\end{proposition}

\subsection{Kan Extensions and the Shapiro Lemma.} \label{sub:KanShapiro}
Let \mbox{$\iota\colon \calc\to \cald$} be any functor. For each
object $d\in\cald$ one defines the overcategory $\iota\da d$ to be
the category with objects $(c,\alpha)$, where $c\in\calc$ and
$\alpha\colon \iota(c)\to d$ is a morphism in $\cald$. Morphisms
from $(c,\alpha)$ to $(c',\alpha')$ in $\iota\da d$ are morphisms
$\gamma\colon c\to c'$ in $\calc$ such that
$\alpha'\circ\iota(\gamma)=\alpha$. The undercategory $d\da\iota$ is
defined analogously. Both categories admit an obvious forgetful
functor to $\calc$, and for $M\in\calc\mod$ we denote the composite of $M$ with the forgetful functor by $M_\sharp$.

The functor \mbox{$\iota^* \colon \cald\mod\to \calc\mod $} induced
by $\iota$ has a right adjoint $R_\iota$ and a left adjoint
$L_\iota$ called the right and left Kan extension along $\iota$, respectively. The Kan extensions of a given functor $M\in\calc\mod$ is determined
objectwise by
\[R_\iota(M)(d) = \lim_{\overcat{d}{\iota}}{M_\#}
\quad\quad\mathrm{and}\quad\quad
L_\iota(M)(d) = \colimnoarrow_{\overcat{\iota}{d}}{M_\#},\]
and a morphism $d \to d'$ in $\cald$ induces homomorphisms
\[R_\iota(M)(d) \to R_\iota(M)(d')
 \quad\quad\mathrm{and}\quad\quad
 L_\iota(M)(d) \to L_\iota(M)(d'),\]
via the universal properties of the limit and colimit functors.
From these descriptions of Kan extensions one easily obtains
the isomorphisms
 \[\lim_{\cald} R_\iota(M) = \lim_{\calc} M
 \quad\quad\mathrm{and}\quad\quad
   \colimnoarrow_{\cald} L_\iota(M) = \colimnoarrow_{\calc} M.\]

For a functor $\iota\colon\calc\to\cald$ and a  system of coefficients $M \in \calc\mod$ there is a homomorphism called the \emph{Shapiro map}
  \[ \Sh_M \colon H^{*}(\cald, R_\iota(M)) \to H^{*}(\calc, M), \]
which is constructed as follows.
Let $M \to I_\bullet$ be an injective resolution of $M$. Since $R_\iota$
is the right adjoint of the exact functor $\iota^*$ it preserves injectives,
and we have a (possibly non-exact) cochain complex
$R_\iota M \to R_\iota I_0 \to R_\iota I_1 \to \cdots$
in which every term after $R_\iota M$ is injective.
If $R_\iota M \to I'_\bullet$ is an injective resolution of $R_\iota M$, then
the identity transformation of $R_\iota M$ lifts (non-uniquely) to
a cochain map $\chi \colon I'_\bullet \to R_\iota I_\bullet$, which induces
a cochain map of limits
 \[ \lim_{\cald}  I'_\bullet \defeq\hom_\cald(\R_\cald, I'_\bullet) \xrightarrow{\chi_*}  \hom_\cald(\R_\cald,R_\iota I_\bullet) \xrightarrow[\cong]{\rho^{-1}} \hom_\calc(\R_\calc,  I_\bullet) \defeq
 \lim_{\calc} I_\bullet, \]
and $\Sh_M$ is defined as the induced map in cohomology.
Here,
\[\rho\colon \hom_\calc(\R_\calc, I_\bullet) = \hom_\calc(\iota^*\R_\cald, I_\bullet) \to\hom_\cald(\R_\cald, R_\iota I_\bullet)\]
is the adjunction isomorphism.
The Shapiro map is independent of the choice of the injective resolutions $I_\bullet$ and
$I'_\bullet$, and the cochain map $\chi$, and is natural with respect to morphisms
$M\to M'$ in $\calc\mod$.
Hence it induces a natural transformation
\[ \Sh \colon H^{*}(\cald,  R_\iota(-)) \to
H^{*}(\calc, -)  \]
which will be referred to as the \emph{Shapiro transformation}. The important property of the Shapiro transformation is given by the following lemma.

\begin{lemma}[The Shapiro Lemma] \label{lem:Shapiro} Let \mbox{$\iota \colon \calc \to
\cald$} be a functor. If the right Kan extension $R_\iota$ is exact
then the Shapiro transformation is a natural isomorphism of functors.
\end{lemma}
\begin{proof}
If $R_\iota$ is exact, then $R_\iota M \to R_\iota I_\bullet$ is an
injective resolution, so we can take $I'_\bullet = R_\iota I_\bullet$
and $\chi = Id$ in the construction of $\Sh_M$ for $M \in \calc\mod$,
and it follows that $\Sh_M$ is an isomorphism.
\end{proof}

The right Kan extension functor is always left exact as it is a
right adjoint. Similarly the left Kan extension functor is right
exact. In some instances the left and right Kan extension functors
are naturally isomorphic. In this case both Kan extension functors
are exact, and in particular the Shapiro Lemma holds. The Shapiro
lemma has an analogous homological version, involving the derived
functors of the colimit, and the left Kan extension.

\subsection{Deformation Retracts of Categories}

We now study a condition on a functor which ensures that it induces a natural
isomorphism between the corresponding derived functors of the limit
and colimit.

\begin{definition}\label{deformation retract} Let $f \colon \calc \to \cald$ be a functor between small categories. We say that $f$ is a \emph{left deformation retract} of $\cald$ if there exists a functor $r \colon \cald \to \calc$ such that $r \circ f = 1_\calc$, and a natural transformation $\eta \colon f \circ r \to 1_\cald$ satisfying $\eta_{f(c)} = 1_{f(c)}$ for $c \in \calc$. When $f$ is the inclusion of a subcategory, we say that $\calc$ is a \emph{left deformation retract of $\cald$}. \emph{Right deformation retracts} are defined dually.
\end{definition}

Notice that this falls short of defining $\iota$ as left adjoint to
$r$ in the sense that we do not require that $r\eta$ is the identity
transformation on $r$. The definition should remind the reader of deformation
retracts of spaces. Indeed, if $\calc$ is a (left or right) deformation
retract of $\cald$, then $|\calc|$ is a deformation retract of $|\cald|$.

\begin{lemma} \label{lem:DefRetLimits}
Let \mbox{$f \colon \calc \to \cald$} be a functor between small categories.
\begin{enumerate}
  \item[(a)]
If $f$ is a left deformation retract then $f$ preserves limits.
  \item[(b)]
If $f$ is a right deformation retract then $f$ preserves colimits.
\end{enumerate}
\end{lemma}
\begin{proof}
It suffices to prove part (a), as part (b) follows by duality, so
assume that $f$ is a left deformation retract and let $r \colon
\cald \to \calc$ and $\eta \colon f \circ r \to 1_\cald$ be as in
Definition \ref{deformation retract}. Since $r \circ f = 1$, we can
regard $f$ as the inclusion of a subcategory. To show that $f$
preserves limits, it therefore suffices (\cite{MacL}) to show that
$f$ is left cofinal in the sense that for every object $d$ in
$\cald$, the overcategory $f\downarrow d$ is connected (meaning it is nonempty and there exists a zigzag of morphisms between any two objects).

First, the map $\eta_d \colon f(r(d)) \to d$ gives rise to an object
$(r(d),\eta_d)$ in $f \downarrow d$. Now, if $(c,u)$ is another
object consisting of an object $c$ in $\calc$ and a morphism $f(c)
\xrightarrow{u} d$ in $\cald$, then the natural transformation $\eta$ gives
rise to the following commutative square.
\[
\xymatrix{
  (f \circ r)(f(c)) \ar[rr]^{(f\circ r)(u)} \ar@{=}[d]^{\eta_{f(c)}} && (f\circ r)(d) \ar[d]^{\eta_d}  \\
  f(c) \ar[rr]^u && d 
}
\]
Thus we have a morphism $r(u) \colon (c,u) \to (r(d),\eta_d)$ in
$f\downarrow d$, and in particular $(c,u)$ is in the same connected
component as $(r(d),\eta_d)$, proving that $f \downarrow d$.
\end{proof}

Lemma \ref{lem:DefRetLimits} actually holds if $f$ is a \emph{weak}
deformation retract, for which we require only that for $c \in
\calc$, we have $\eta_{f(c)} = f(h)$ for some morphism $h$ in
$\calc$. Using the full strength of deformation retracts, one can
prove a stronger result, namely that if $f$ is a left deformation
retract then it induces an isomorphism on cohomology, and if it is a
right deformation retract then it induces an isomorphism on homology
(Lemma \ref{lem:DefRetLimits} makes that claim only in dimension 0,
which is all we need for our purposes).

\subsection{Cohomology with Locally Constant Coefficients}
\label{loc-const-coeff}
\newcommand{\kk}{\mathrm{R}}

We now discuss an important  special case of functor cohomology, namely the case where the coefficient system is locally constant.

Let $\calc$ be a small connected category, and let $\widehat{\calc}$ denote the groupoid completion of $\calc$, i.e., the category with the same objects as $\calc$, and
where all morphisms are formally inverted. Let $\pi\colon\calc\to \widehat{\calc}$ be the obvious functor.
Choose an object $c_0\in\calc$. For any other object $c\in\calc$, choose a morphism $\phi_c\in\mor_{\widehat{\calc}}(c,c_0)$,
and take $\phi_{c_0}$ to be the identity.  Then one gets isomorphisms of sets
\[j_{a,b}\colon\mor_{\widehat{\calc}}(a,b) \to \aut_{\widehat{\calc}}(c_0)\quad\mathrm{by}\quad \varphi \mapsto\phi_b\circ
\varphi\circ\phi_a^{-1}.\]
Let $\Gamma\defeq\aut_{\widehat{\calc}}(c_0)$, and let $\iota\colon\calb(\Gamma)\to\widehat{\calc}$ be the inclusion functor.
Assembling the isomorphisms $j_{a,b}$ together for all pairs of objects in $\widehat{\calc}$, one gets a functor
\[J\colon \widehat{\calc}\to \calb(\Gamma),\]
which is an equivalence of categories (but depends on the choices made). In particular, $J\circ\iota$ is the identity on
$\calb(\Gamma)$, while $\iota\circ J$ is naturally isomorphic to the identity on $\widehat{\calc}$. Moreover, by \cite[Prop. 1]{Qu},
the group $\Gamma$ is naturally isomorphic to $\pi_1(|\calc|, c_0)$, and the map induced on nerves by the composite $J\circ\pi$
induces an isomorphism on fundamental groups.

A system of coefficients $M\in\calc\mod$ is said to be \emph{locally constant} if for every morphism $\varphi\colon a\to b$ in $\calc$,
$M(\varphi)$ is an isomorphism. Notice that $M$ is locally constant if and only if $M$ factors  as $M = \widehat{M}\circ\pi$ for
a unique functor $\widehat{M}\in\widehat{\calc}\mod$.
Let $\kk$ be a commutative ring with a unit. An $\kk[\Gamma]$-module  is  a functor $N \colon \calb(\Gamma) \to \kk\mod$. The next
lemma shows that every locally constant system of coefficients on $\calc$ is, up to a natural isomorphism, a $\kk[\Gamma]
$-module composed with the functor $J$ constructed above.

\begin{lemma}\label{loc-const-factorization}
Let $\calc$ be a small connected category, and let $M\in\calc\mod$ be a locally constant system of coefficients on $\calc$.
Then $M$ is naturally isomorphic  to $N\circ J\circ\pi$, where $N$ is the $\kk[\Gamma]$-module $\widehat{M}\circ\iota$.
Moreover, this defines an equivalence of categories between the category of $\kk[\Gamma]$-modules and the category of
locally constant functors on $\calc$.
\end{lemma}
\begin{proof}
Since $\iota\circ J$ is naturally isomorphic to the identity on $\widehat{\calc}$, one has
\[M = \widehat{M}\circ\pi \cong \widehat{M}\circ \iota\circ J\circ \pi = N\circ J\circ\pi,\]
as claimed.

For the second statement, notice that the correspondences which takes an $\kk[\Gamma]$-module $N$ to
$N\circ J\circ\pi$ and a locally constant system of coefficients $M\in\calc\mod$ to $\widehat{M}\circ\iota$
are natural on $N$ and $M$ respectively, and using the first statement, define the equivalence of categories claimed.
\end{proof}

\begin{proposition}\label{loc-const-coho}
Let $\calc$ be a small connected category, and let $M\in\calc\mod$ be a locally constant functor.
Then, with the notation above,  there is an isomorphism
\[H^*(\calc, M) \cong H^*(|\calc|, \widehat{M}\circ\iota)\]
which is natural in $M$.
\end{proposition}
\begin{proof}
As a consequence of  \cite[Prop 1]{Qu}, Quillen shows that if $\calc$ is a small connected category, and $L$ is a $\Z[\pi_1(|\calc|)]$-module, then there is a canonical isomorphism 
\[H^*(|\calc|, L) \cong H^*(\calc, L'),\]
where $L'\in\calc\mod$ is given by the composite
\[\calc\xrightarrow{\pi}\widehat{\calc}\xrightarrow{J}\calb(\pi_1(|\calc|))\xrightarrow{L}\Ab.\] But by Lemma \ref{loc-const-factorization}, every locally constant $M\in\calc\mod$  is naturally isomorphic to a functor of the form $M\circ\iota\circ J\circ\pi$. Thus, if we set $L = \widehat{M}\circ\iota$, then Quillen's result reads
\[H^*(|\calc|, \widehat{M}\circ\iota)  \cong H^*(\calc, M\circ\iota\circ J\circ\pi) \cong H^*(\calc, M).\]
This proves the claim.
 \end{proof}

\begin{corollary}\label{nerves=>cohomology}
Let $\tau\colon \calc\to\cald$ be a functor between small connected
categories, such that $|\tau|$ is a homotopy equivalence. Then for
any locally constant system to coefficients $M\in\cald\mod$, $\tau$
induces a natural isomorphism
\[H^*(\calc, \tau^*M) \xrightarrow{\cong} H^*(\cald, M).\]
\end{corollary}

\subsection{Covering Spaces, and Locally Constant Coefficients}
\label{cov-sp-loc-cosnt-coeff}
 Let $\calc$ be a small connected
category, $c_0\in\calc$ and
$\Gamma=\aut_{\widehat{\calc}}(c_0)\cong\pi_1(|\calc|,c_0)$. Let
$J\colon\calc\to\calb(\Gamma)$ be as before.

For each $\Gamma'\le\Gamma$, let $\cale_\Gamma(\Gamma/\Gamma')$ denote the category with $\Gamma/\Gamma'$ as objects,
and with a unique morphism $\hat{g}\colon a\Gamma'\to ga\Gamma'$ for each $g\in\Gamma$ and $a\Gamma' \in\Gamma/\Gamma'$.
Let $\calc_{\Gamma'}$ denote the category given by the pull back in the diagram
\[\xymatrix{
\calc_{\Gamma'} \ar[rr] \ar[d]_{\pi}  && \cale_\Gamma(\Gamma/\Gamma') \ar[d]^{\rho_{\Gamma'}}\\
\calc   \ar[rr]_{J} && \calb{\Gamma}
} \]
Thus, $\Obj(\calc_{\Gamma'}) = \Obj(\calc)\times \Gamma/\Gamma'$,
while morphisms $(c, a\Gamma')\to (c',a'\Gamma')$ are morphisms
$\varphi\colon c\to c'$, such that $a'\Gamma' = J(\varphi)a\Gamma'$.

For each $c\in\calc$ the undercategory $c\downarrow \pi$ has objects
$((d,a\Gamma'),\varphi)$ where $d\in\calc$,
$a\Gamma'\in\Gamma/\Gamma'$, and $\varphi\colon c\to d$ is a
morphism in $\calc$. A morphism
$((d,a\Gamma'),\varphi)\to((d',a'\Gamma'),\varphi')$ in
$c\downarrow\pi$ is a map $\psi\colon d\to d'$ in $\calc$, such that
$\psi\varphi = \varphi'$, and  $a'\Gamma' = J(\psi)a\Gamma'$.

Let $((d,a\Gamma'),\varphi)$ be an arbitrary object in
$c\downarrow\pi$. Then there is a unique morphism
$((c,a\Gamma'),1_c)\to ((d,J(\varphi)a\Gamma'),\varphi)$ which is induced by $\varphi$.
Thus every component of $c\downarrow\pi$ has an initial object and
is therefore contractible. Notice that there is an obvious 1--1
correspondence between components of $c\downarrow\pi$ and
$\Gamma/\Gamma'$, and that every morphism $c\to e$ in $\calc$
induces a homotopy equivalence $e\downarrow\pi\to c\downarrow\pi$.
Thus the hypothesis of Quillen's theorem B are satisfied, and
$|\calc_{\Gamma'}|\to|\calc|$ is a covering space up to homotopy,
with fibre over $c$ given by $|c\downarrow\pi|\simeq
\Gamma/\Gamma'$.

\begin{lemma}\label{categorical covering}
Let $\tau\colon\cald\to\calc$ be a functor between small connected
categories such that the induced map $|\tau|\colon|\cald|\to|\calc|$
is a covering space up to homotopy (i.e., has a homotopically
discrete homotopy fibre), and let $\Gamma'=\pi_1(|\cald|, d_0)$ for
some $d_0\in\cald$. Then the following hold: \begin{enumerate}
\item[(i)]There is a functor
$\widehat{\tau}\colon\cald\to\calc_{\Gamma'}$, lifting $\tau$ (i.e.,
$\pi\widehat{\tau}=\tau$), which induces a homotopy equivalence on
nerves.
\item[(ii)] For any locally constant system of coefficients
$M\in\calc\mod$, there is a natural isomorphism
$R_\tau(\tau^*M)\cong R_\pi(\pi^*M)$.
\end{enumerate}
\end{lemma}
\begin{proof}
Choose some $d_0\in\cald$, and let $c_0=\tau(d_0)$.  Let $\Gamma = \pi_1(|\calc|,c_0)$, and $\Gamma'=\pi_1(|\cald|,d_0)$.
Then, constructing a projection functor $J\colon\calc\to \calb(\Gamma)$ as before,
one gets $J\circ\tau\colon\cald\to \calb(\Gamma)$, whose image is $\Gamma'$, and thus a
corresponding projection $J'\colon\cald\to\calb(\Gamma')$. The diagram
\[\xymatrix{
\cald \ar[rr]^{\tau}\ar[d]_{J'} && \calc \ar[d]^{J} && \calc_{\Gamma'}\ar[ll]_{\pi}\ar[d]\\
\calb(\Gamma')  \ar[rr]^{inc} && \calb(\Gamma)
&& \ar[ll]_{\rho_{\Gamma'}}\cale_\Gamma(\Gamma/\Gamma')
}\] 
commutes, where the right hand side square is a pull back square. The inclusion of
$\calb(\Gamma')$ in $\cale_\Gamma(\Gamma/\Gamma')$ as the full
subcategory on the object $1\Gamma'$ induces a homotopy equivalence
on nerves, and  a functor
$\widehat{\tau}\colon\cald\to\calc_{\Gamma'}$. A simple diagram
chase now shows that $|\widehat{\tau}|$ is a homotopy equivalence,
and proves (i).

To prove (ii), notice that since $\pi\widehat{\tau} = \tau$, the
functor $\widehat{\tau}$ induces a natural transformation of
functors $\calc\to\cat$
\[\widehat{\tau}_{(-)}\colon (-\da\tau)\to (-\da\pi).\]
Thus for every $c\in\calc$, and $M\in\calc\mod$, one has a map
induced by $\widehat{\tau}$,
\[R_\pi(\pi^*M)\defeq\lim_{c\da\pi}(\pi^*M)_\sharp\to\lim_{c\da\tau}(\widehat{\tau}^*\pi^*M)_\sharp = \lim_{c\da\tau}(\tau^* M)_\sharp\defeq R_\tau(\tau^*M).\]
But since $\pi^*M$ is also locally constant, and since
$\widehat{\tau}$ induces a homotopy equivalence on nerves, it
follows from Corollary \ref{nerves=>cohomology} that this map is an
isomorphism, thus proving (ii).
\end{proof}

Lemma  \ref{categorical covering} allows us to compute the values of a right Kan extension of a system of
coefficients $M\in\cald\mod$ along a functor $\tau\colon\cald\to \calc$ satisfying its hypothesis.

The following lemma allows a calculation of the right Kan extension explicitly in a specialized case which will be of interest in our discussion.

\begin{lemma}\label{RKE along coverings}
Let $\tau\colon\cald\to\calc$ be a functor between small connected categories such that for each $c\in\calc$,
every connected component in the under category $c\downarrow\tau$ has an initial object. Then for  any $M\in\cald\mod$,
the value of the right Kan extension $R_\tau(M)\in\calc\mod$ at $c$ is the direct product of the values of $M_\sharp$
on these initial objects. In particular $R_\tau$ is an exact functor.
\end{lemma}
\begin{proof}
For each $c\in\calc$ one has
\[R_\tau(M)(c) \defeq \lim_{c\downarrow\tau}M_\sharp\cong \prod_{[(d,\alpha)]\in[c\downarrow\tau]}M(d_0),\]
where the product runs over all connected components of $c\downarrow\tau$, and $(d_0,\alpha_0)$ is initial in the
component of $(d,\alpha)$. This follows at once from the fact that a limit over a category with an initial object
is given by the value of the functor on that object. This description also makes it clear that $R_\tau$ is an exact
functor, and so the proof is complete.
\end{proof}

\begin{proposition}\label{RKE along hty covering}
Let $\tau\colon\cald\to\calc$ be a functor between small connected categories such that the induced map
$|\tau|\colon|\cald|\to|\calc|$ is a covering space up to homotopy. 
Fix an object $d_0$ in $\cald$, let $\Gamma' = \pi_1(|\cald|, d_0)$, and let $\pi\colon\calc_{\Gamma'}\to \calc$ be the projection.

For any locally constant system of coefficients $M\in\calc\mod$, there is a natural isomorphism 
\[H^*(\calc, R_\pi(\pi^*M)\in\calc\mod) \cong H^*(\cald, \tau^*M). \]
Furthermore, if the index of the covering is finite, then there is a natural transformation $T\colon R_\pi(\pi^*M)\in\calc\mod \to M$,
such that the composite
\[H^*(\calc, M)\xrightarrow{\tau^*} H^*(\cald, \tau^*(M))\cong H^*(\calc,R_\pi(\pi^*M)\in\calc\mod)\xrightarrow{T_*} H^*(\calc, M)\]
is multiplication by the index.
\end{proposition}
\begin{proof}
Since $M$ is locally constant
$H^*(\cald, \tau^*M)\cong H^*(\calc_{\Gamma'}, \pi^*M)$ by Corollary \ref{nerves=>cohomology}.
By Lemma  \ref{RKE along coverings} $R_\pi$ is exact,
and hence we have an isomorphism
$H^*(\calc, R_\pi(\pi^*M)) \cong H^*(\calc_{\Gamma'}, \pi^*M)$ by the Shapiro lemma. This proves the first statement.

To prove the second statement, consider the functor $R_\pi(\pi^*M)$. For each coset $a\Gamma'\in\Gamma/\Gamma'$, the
object $((c,a\Gamma'),1_c)\in c\downarrow\pi$ is initial in its own path component, and $(\pi^*(M))_\sharp((c,a\Gamma'),1_c) = M(c)$.
Thus, by Lemma \ref{RKE along coverings}, for each $c\in\calc$,
\[R_\pi(\pi^*M)(c) \defeq \lim_{c\downarrow\pi}\pi^*M_\sharp \cong \prod_{\Gamma/\Gamma'}M(c).\]
Assuming $\Gamma/\Gamma'$ is a finite set, define a natural
transformation of functors  $T\colon R_\pi(\pi^*M)\to M$ in $\calc\mod$ by
\begin{equation}\label{T}
T_c(\{x_{a\Gamma'}\}_{a\Gamma'\in\Gamma/\Gamma'}) =
\sum_{a\Gamma'\in\Gamma/\Gamma'}x_{a\Gamma'}.\end{equation}
Naturality of $T$ is clear, since $M$ is additive.

It remains to show that $T_*\circ\tau^*$ is multiplication by $|\Gamma\colon\Gamma'|$. To do that it suffices to show that the natural transformation induced by $\pi$
\[\hom_{\calc\mod}(-,M)\xrightarrow{\pi^*}\hom_{\calc_{\Gamma'}}(\pi^*(-), \pi^*M)\cong\hom_{\calc\mod}(-,R_\pi(\pi^*M)),\]
is the diagonal map. But for $U\in\calc\mod$, a natural transformation $\eta\colon U\to M$ takes an object $c\in\calc$ to a morphism  $\eta_c\colon U(c)\to M(c)$. Composing with $\pi$ we see that the composition transformation takes each object $(c,a\Gamma')$ in $\calc_{\Gamma'}$ to the morphism $\eta_c$ as above.
 This shows that $\pi^*$ is precisely the diagonal map and the proof is complete.
\end{proof}

The next corollary follows at once from the definitions.
\begin{corollary}\label{geometric}
Let $M\in\calc\mod$ be a locally constant system of coefficients, and let $\Gamma'\le\Gamma = \pi_1(|\calc|,c_0)$ be a subgroup of finite index. Then the map
\[H^*(\calc_{\Gamma'},\pi^*M)\cong H^*(\calc,R_\pi(\pi^*M))\xrightarrow{T_*}H^*(\calc,M)\]
coincides with the geometric transfer map associated to the covering space $|\calc_{\Gamma'}|\to|\calc|$.
\end{corollary}


\section{$p$-Local Finite Group Cohomology}
\label{Cohomology with coefficients} The cohomology of a $p$-local
finite group $\SFL$ with coefficients $M\in\call\mod$ is the main
object of study in this paper. In ordinary finite group cohomology
the  transfer map with respect to a subgroup is one of the most
useful computational and theoretical tools available to us, and
its properties are the key to proving the stable element theorem
in cohomology, among other things. We start our discussion of the
$p$-local analog with an example showing
that in $p$-local finite group cohomology one cannot associate
a transfer map to an arbitrary $p$-local subgroup inclusion,
even if one restricts to locally constant coefficients.

\subsection{An Example}
As already mentioned, in \cite{BLO2} it is shown that if $\SFL$ is a
$p$-local finite group then the mod $p$ cohomology of $|\call|$ is
given by the $\calf$-stable elements in $H^*(BS,\Fp)$. The next
statement, due to Bob Oliver, provides an abundance of examples that
show that with our definition of local coefficients one cannot
expect a stable elements theorem to hold in full generality.

\begin{proposition}\label{Bob's counter example}
Let $\SFL$ be a $p$-local finite group, (with $S$ nontrivial) and
let $\Gamma=\pi_1(|\call|)$. Fix a compatible system of inclusions,
and assume that the composite functor
 \[ \calb(j) \colon \calb(S) \xrightarrow{\iota} \call \xrightarrow{J} \calb(\Gamma) \]
is faithful, where $\iota$ is the inclusion $\calb(S)
\xrightarrow{\calb(\delta_S)} \calb(\aut_{\call}(S)) \subseteq \call$, and
$J$ is the functor defined in Section \ref{sub:InclRes}. Assume
further that the universal cover $\widetilde{|\call|}$ is not mod
$p$ acyclic. Let $M$ be the group ring $\Fp[\Gamma]$ regarded as a
system of coefficients in $\call\mod$ via $J$ and the obvious
$\Gamma$  action. Then $\widetilde{H}^*(S,\iota^*M)=0$, but
$\widetilde{H}^*(\call, M )\neq 0$.
\end{proposition}
\begin{proof}
Since $S \xrightarrow{j} \Gamma$ is monic, $\iota^*M$ is a projective
$\Fp[S]$ module, and since $S$ is finite, it is also injective,
and thus acyclic. On the other hand, since $M$ is locally constant,
\begin{multline*}
H^*(\call, M) = H^*(|\call|,\Fp[\Gamma]) \defeq
H^*(\hom_{\Fp[\Gamma]}(C_*(\widetilde{|\call|}), \Fp[\Gamma])) \cong\\
H^*(\hom_{\Fp}(C_*(\widetilde{|\call|}), \Fp)) \cong
H^*(\widetilde{|\call|},\Fp),\end{multline*} where the first
equality follows from Proposition  \ref{loc-const-coho}, and the
right hand side is not mod $p$ acyclic by assumption.
\end{proof}

One large family of examples which satisfy the conditions of
Proposition \ref{Bob's counter example} is the general linear groups
$GL_n(\FF_{p^k})$. The following argument is a sketch of a proof.
Since the same argument applies for all $k$, we will replace
$\FF_{p^k}$ by $\FF$, while $p$ and $k$ are assumed fixed. The
radical subgroups in these groups are all $p$-centric, and their
centers are given by the diagonally embedded $\FF^*\cong\Z/(p^k-1)$.
Hence the centric radical linking systems of $GL_n(\FF)$ and the
associated projective group $PGL_n(\FF)$ have the same objects, and
$|\call^{cr}(GL_n(\FF))|$ is a covering space of
$|\call^{cr}(PGL_n(\FF))|$ with fibre $\FF^*$, and thus they share
the same universal cover. But $\call^{cr}(PGL_n(\FF))$ coincides
with the centric radical transporter system of $PGL_n(\FF)$ (the
category with the same objects, and where morphism sets are the
transporters $N_G(P,Q)$), and thus admits an obvious map
\[|\call^{cr}(PGL_n(\FF))|\to BPGL_n(\FF),\]
with fibre given by the nerve of the poset of centric radical
subgroups in $PGL_n(\calf)$ \cite{BLO1}. This nerve is known as the
Tits building, and has the homotopy type of a wedge of spheres.
Thus it is the universal cover of $|\call^{cr}(PGL_n)|$, while
at the same time it is  not acyclic. This was first observed by
Grodal. To finish the argument, we recall from \cite{BCGLO1} that
for any $p$-local finite group, the centric radical linking system
and the centric linking system have homotopy equivalent nerves.

\subsection{Transfer Maps for Locally Constant Coefficients}
\label{Loc-Const-coeff} We now show that if $\SFL$ is a $p$-local
finite group and $(S_0,\calf_0,\call_0)$ is a subgroup of $p$-power
index or index prime to $p$, then one can define an algebraic
transfer map, which coincides with the transfer map associated to
the finite covering $|\call_0|\to|\call|$ (See Thm.
\ref{Classification-summary}(iv)).

First we must discuss restrictions of coefficient systems.
We defined a system of coefficients on a
$p$-local finite group to be a functor $M\in\call\mod$. When $(S_0,\calf_0,\call_0)$ is a subgroup of $\SFL$ of $p$-power
index, one does not have in general an inclusion functor
$\call_0\to\call$, but rather an inclusion $\call_0^q\to\call^q$. Thus,  in this case, a system of
coefficients on $\SFL$ cannot be directly restricted to
$(S_0,\calf_0,\call_0)$. We will show that for locally constant coefficients, this is in fact an easily solvable problem. We start by observing that the restriction of a locally constant system of coefficients on $\call^q$ to $\call$ does not affect cohomology.

\begin{lemma} \label{H(Lq)=H(L)}
Let $\SFL$ be a $p$-local finite group, and let $\call^q$ be the
associated quasicentric linking system. For every locally constant
system of coefficients $N\in\call^q\mod$, the restriction to $\call$
induces an isomorphism in cohomology $H^*(\call^q, N)\cong
H^*(\call, \iota^*N)$.
\end{lemma}
\begin{proof}
This follows directly from Corollary
\ref{nerves=>cohomology} since the inclusion $\call^q \subseteq \call$
induces a homotopy equivalence on nerves.
\end{proof}

Below we will outline a functorial way to extend a locally constant coefficient
system $M$ on $\call$ to a locally constant coefficient system $M^q$
on $\call^q$. The extension $M^q$ can then be restricted to a locally
constant coefficent system $M_0$ on $\call_0$ via the inclusions
$\call_0 \subseteq \call^q_0 \subseteq \call^q$. We will furthermore
show that the extension $M^q$ is unique up to unique isomorphism extending
the identity morphism of $M$, and thus it makes sense to think of $M_0$ as
the restriction of $M$ to $\call_0$. Lemma \ref{H(Lq)=H(L)} shows that this
convention has the desired effect in cohomology.

Given a locally constant system of coefficients $M$ on $\call$, we construct
an extension $M^q$ of $M$ to $\call^q$ as follows. First recall that $M$ extends uniquely to a system of coefficients $\widehat{M}$ on the groupoid completion $\widehat{\call}$. Since every morphism in $\call^q$ is both a monomorphism and an epimorphism in the categorical sense (\cite[Corollary 3.10]{BCGLO1}), the inclusion $i \colon \call \to \call^q$ induces an inclusion of groupoid completions, $\widehat{i} \colon \widehat{\call}\to\widehat{\call^q}$. Since $i$ induces a homotopy equivalence on nerves, and in particular an isomorphism of fundamental groups, $\widehat{i}$ is an equivalence of categories.
This implies that $\widehat{\call}$ is a deformation retract of $\widehat{\call^q}$, and we can choose an inverse $r\colon\widehat{\call^q}\to\widehat{\call}$ such that $r \circ \widehat{i} = id_{\widehat{\call}}$.  Now the composite
 \[ M^q \colon \call^q\xrightarrow{\pi}\widehat{\call^q}\xrightarrow{r}\widehat{\call}\xrightarrow{\widehat{M}}
\R\mod, \]
where $\pi$ is the groupoid completion map, is a locally constant system of coefficients on $\call^q$ that extends $M$.

Fixing a choice for the retraction $r$, the construction of $M^q$ as the restriction of $\widehat{M}$ along $r\circ \pi$ clearly makes the assignment $M \mapsto M^q$ functorial. However this extension functor is by no means unique or even canonical as a different choice of retract would give rise to a different extension functor. The next lemma shows that, while there are many such extension functors, the difference between them is inconsequential.

\begin{lemma}
Let $\SFL$ be a $p$-local finite group, let $\call^q$ be the
associated quasicentric linking system, and let $M$ be a locally
constant system of coefficients on $\call$. If $M'$ is a locally
constant system of coefficients on $\call^q$ extending $M$, then
there is a unique natural isomorphism $\eta \colon M^q \Rightarrow M'$
that restricts to the identity transformation of $M$.
\end{lemma}
\begin{proof}
The isomorphism $\eta$ sends $P$ to the composite
 \[ \eta_P \colon M^q(P) \xrightarrow{M^q(\iota_P)} M^q(S) = M(S) = M'(S) \xrightarrow{M'(\iota_P)^{-1}} M'(P). \]
The verification of naturality and uniqueness of $\eta$ is routine and is left for the reader.
\end{proof}

In light of this uniqueness result, we make the following definition.

\begin{definition}
Let $\SFL$ be a $p$-local finite group and let $\Gamma = \Gamma_p(\calf)$ or $\Gamma_{p'}(\calf)$.
Let $H\le \Gamma$, let $(S_H, \calf_H,\call_H)$ be the corresponding $p$-local subgroup,
and let $\iota^q \colon\call^q_H\to\call^q$ be the inclusion.
For a locally constant system of coefficients $M$ on $\call$, the
\emph{restriction of $M$ to $\call_H$} is the composite
 \[ M_H  \colon \call_H \subseteq \call^q_H \subseteq \call^q \xrightarrow{M^q} \R\mod. \]
The \emph{restriction map} $\Res_H \colon H^*(\call,M) \to H^*(\call_H,M_H)$ is the
unique map that fits into the commutative diagram
 \[ \xymatrix{
   H^*(\call^q,M^q) \ar[d]^\cong \ar[rr]^{\iota^*} && H^*(\call^q_H,\iota^*M^q) \ar[d]^\cong \\
   H^*(\call,M) \ar[rr]^{\Res_H} && H^*(\call_H,M_H),
 }
 \]
where the vertical maps are the isomorphisms from Lemma \ref{H(Lq)=H(L)}.
\end{definition}

When $\Gamma = \Gamma_{p'}(\calf)$, there is a well-defined inclusion $\call_H \subset \call$,
and we note that $M_H$ is just restriction along this inclusion, and $\Res_H$ is the usual
restriction map.

The following proposition holds for both centric and quasicentric linking systems,
although we state it using the notation for centric linking systems.

\begin{proposition}\label{plf-subgrps-loc-const-coeff}
Let $\SFL$ be a $p$-local finite group and let $\Gamma = \Gamma_p(\calf)$ or $\Gamma_{p'}(\calf)$.
Let $H\le \Gamma$, and let $(S_H, \calf_H,\call_H)$ be the corresponding $p$-local subgroup. For
any locally constant system of coefficients $M\in\call\mod$ there exists a map
\[Tr_H\colon H^*(\call_H, M_H)\to H^*(\call, M),\]
such that $Tr_H \circ \Res_H$ is multiplication by $|\Gamma\colon H|$
on $H^*(\call, M)$. Furthermore, the map thus defined coincides with
the transfer map associated with the covering $|\call_H|\to|\call|$.
\end{proposition}
\begin{proof}
It suffices to prove this in the quasicentric case, as the centric
case is obtained by restriction. Let $\iota \colon \call^q_H \to \call^q$
be the inclusion. By Theorem \ref{Classification-summary},
the map $|\iota|\colon|\call_H^q|\to|\call^q|$ is a covering space up to
homotopy, with homotopy fibre equivalent to $\Gamma/H$. Thus
Proposition \ref{RKE along hty covering} applies and one has a
system of coefficients $M'\in\call\mod$, such that
\[H^*(\call^q, M') \cong H^*(\call^q_H, \iota^*M),\]
and a natural transformation $T\colon M'\to M$ such that the
composite
\[H^*(\call^q, M)\xrightarrow{\iota^*} H^*(\call^q_H, \iota^*M) \cong H^*(\call^q,
M')\xrightarrow{T_*} H^*(\call^q, M)\] is multiplication by the index
$|\Gamma\colon H|$. Identifying the two middle terms in the sequence
via the isomorphism between them, define $Tr$ to be the map $T_*$.

The second statement follows at once from Corollary \ref{geometric}.
\end{proof}

\subsection{Subsystems of Index Prime to $p$}
We now restrict attention to subgroups of $p$-local finite groups of index prime to $p$.
It is in this context that we are able to obtain our most general results.
However some of the discussion here has more general implications as well.

Throughout this subsection, fix a $p$-local finite group $\SFL$ and
let $\Gamma = \Gamma_{p'}(\calf)$. Fix a subgroup $H \leq \Gamma$
and let $\iota\colon\call_H\to\call$ be the inclusion We assume
throughout that a compatible system of inclusions for $\call$ has
been chosen, and by \emph{restriction} of a morphism in $\call$ we
always mean restriction with respect to the chosen inclusions (see
Subsection \ref{sub:InclRes}). Let $\widehat{\Theta}\colon \call\to
\calb(\Gamma)$ denote the projection, as in Subsection 1.5. If $\alpha$
is a morphism in $\call$, we denote by $[\alpha]$ the class of
$\alpha$ in $\calf$.

Recall that a category is a \emph{discrete groupoid} if each of its morphism sets either consists of a single isomorphism or is empty.

\begin{lemma}\label{RetractionR}
For $P\in\call$, let $\calg_H(P) \subseteq P\downarrow\iota$ and $\calg^H(P) \subseteq
\iota\downarrow P$ denote the full subcategories, whose objects, respectively, are of
the form $(P^\alpha,\alpha)$ and $(P^\alpha,\alpha^{-1})$ for some $\alpha\in\mor_{\call}(P,S)$.
Then $\calg_H(P)$ and $\calg^H(P)$ are discrete
groupoids. Furthermore, $\calg_H(P)$ is a left deformation retract of
$P \downarrow \iota$, and $\calg^H(P)$ is a right deformation retract of $\iota\downarrow P$.
\end{lemma}
\begin{proof}
Since $\call_H$ and $\call$ have the same objects by Theorem
\ref{Classification-summary}, the category $P\downarrow \iota$ is
nonempty for any $P\in\call$. Thus let
$\alpha,\beta\in\hom_{\call}(P,S)$ be any morphisms. Then
$(P^\alpha,\alpha)$ and $(P^\beta, \beta)$ are objects in
$\calg_H(P)$. A morphism between them is a morphism $\varphi\colon
P^\alpha\to P^\beta$ in $\call_H$, such that
$\varphi\circ\alpha=\beta$. But $\alpha\colon P\to P^\alpha$ is an
isomorphism in $\call$, and hence invertible. Thus $\varphi =
\beta\circ\alpha^{-1}$, and so if a morphism between two given
objects exists, then it is unique, and is obviously an isomorphism.
This shows that $\calg_H(P)$ is a discrete groupoid. An analogous
argument shows that $\calg^H(P)$ is also a discrete groupoid.

Let $j\colon\calg_H(P)\to P\downarrow\iota$ denote the inclusion,
and define $r\colon P\downarrow\iota\to\calg_H(P)$ as follows. On
objects $r(Q,\alpha)
\defeq (P^\alpha,\alpha)$, noting that $P^\alpha\le Q$, while any morphism
$(Q,\alpha)\xrightarrow{\varphi}(Q',\alpha')$ is taken by $r$ to
$\alpha'\circ\alpha^{-1}$. For this to be well defined, $\alpha'\circ\alpha^{-1}$
should be a morphism in $\call_H$, and this is indeed the case since
$\varphi \circ \alpha' = \alpha$ and $\widehat{\Theta}(\varphi) \in
H$ implies $\widehat{\Theta}(\alpha'\circ\alpha^{-1}) \in H$. The
composite $r\circ j$ is the identity functor on $\calg_H(P)$, and there
is a natural transformation $\eta\colon j\circ r \to\Id_{P\downarrow\iota}$
that takes an object $(Q,\alpha)$ in
$P\downarrow\iota$ to the morphism $(P^\alpha,\alpha)\to(Q,\alpha)$ induced by the inclusion. Clearly $\eta j$ is the identity
transformation of $j$. This shows that $\calg_H(P)$ is a left
deformation retract of $P\downarrow \iota$, as claimed.

Next, let $j$ denote the inclusion $\calg^H(P)\to \iota\downarrow P$.
For an object $(Q,\alpha)$ in $\iota \downarrow P$ we choose a morphism
$\gamma_{(Q,\alpha)} \in\mor_{\call}(P,S)$ with $\widehat{\Theta}(\gamma_{(Q,\alpha)}) = \widehat{\Theta}(\alpha)^{-1}$ as follows. If $\alpha$ is an isomorphism, then
we just take $\gamma_{(Q,\alpha)} = \alpha^{-1}$. Otherwise, we choose, using Theorem
\ref{Classification-summary}(v), a morphism $\bar{\gamma}_{(Q,\alpha)} \in \aut_\call(S)$
such that
$\widehat{\Theta}(\bar{\gamma}_{(Q,\alpha)}) = \widehat{\Theta}(\alpha)^{-1}$,
and we let $\gamma_{(Q,\alpha)}$ be the restriction of $\bar{\gamma}_{(Q,\alpha)}$
to $P$. We now define the retraction $r \colon \iota\downarrow P \to \calg^H(P)$ by
setting $r(Q,\alpha) = (P^{\gamma_{(Q,\alpha)}},\gamma_{(Q,\alpha)}^{-1})$, and sending
a morphism $(Q,\alpha) \xrightarrow{\varphi} (R,\beta)$ to the unique morphism
$\gamma_{(R,\beta)} \circ \gamma_{(Q,\alpha)}^{-1} \colon (P^{\gamma_{(Q,\alpha)}},\gamma_{(Q,\alpha)}^{-1}) \to (P^{\gamma_{(R,\beta)}},\gamma_{(R,\beta)}^{-1})$ in $\calg^H(P)$.
Observe that this is well defined since
\[ \widehat{\Theta}(\gamma_{(R,\beta)} \circ \gamma_{(Q,\alpha)}^{-1}) =\widehat{\Theta}(\beta)^{-1} \widehat{\Theta}(\alpha) = \widehat{\Theta}(\varphi) \in H,\]
and so $\gamma_{(R,\beta)} \circ \gamma_{(Q,\alpha)}^{-1}$ is in $\call_H$.
With this definition we have $r\circ j = \Id_{\calg^H(P)}$. Furthermore, if
$\eta \colon \Id_{\iota\downarrow P} \to j\circ r$ is the natural transformation
that sends an object $(Q,\alpha)$ to the morphism $\gamma_{(Q,\alpha)} \circ \alpha$,
then $\eta j$ is the identity transformation of $j$, and so $\calg^H(P)$ is a right
deformation retract of $\iota \downarrow P$.
\end{proof}

Recall that the homomorphism induced by the restriction of
$\widehat{\Theta}\colon \callq\to \calb(\Gamma)$ to
$\aut_{\call}(S)$ is surjective by Theorem \ref{Classification-summary}. Thus the
composite with the projection to the right cosets
\begin{equation}
\label{aut2cosets}\aut_{\call}(S)\to \Gamma \to \Gamma/H
\end{equation}
is also surjective.

\begin{lemma}\label{CH(P)-groupoid}
Consider the set of left cosets $\Gamma/H$ as a category with
objects $\Gamma/H$ and no nontrivial morphisms. Then for any
$P\in\call$  the functor
\[\delta_P\colon\calg_H(P)\to \Gamma/H\]
taking an object $(P^\alpha,\alpha)$ to the coset
$\widehat{\Theta}(\alpha)^{-1}H$ is an equivalence of categories.
\end{lemma}
\begin{proof}
Recall that an object in $\calg_H(P)$ is of the form
$(P^\alpha,\alpha)$ where $\alpha\colon P \to P^\alpha \le S$ is an
isomorphism in $\call$. Surjectivity of $\delta_P$ on objects
follows from the surjectivity of $\widehat{\Theta} \colon
\aut_\call(S) \to \Gamma$.

Two objects $(P^\alpha,\alpha)$ and $(P^\beta,\beta)$ are in the
same connected component of $\calg_H(P)$ if and only if
$\beta\circ\alpha^{-1}\colon P^\alpha\to P^\beta$ is a morphism in
$\call_H$. This is the case if and only if
$\widehat{\Theta}(\beta\circ\alpha^{-1})\in H$, or equivalently if and only if $\widehat{\Theta}(\alpha^{-1})H = \widehat{\Theta}(\beta^{-1})H$, or in other words if and only if $\delta_P(P^\alpha,\alpha) =
\delta_P(P^\beta, \beta)$. This shows both that $\delta_P$ is well
defined and that it is an equivalence of categories as claimed.
\end{proof}

We next observe that calculating limits over a discrete groupoid is particularly easy.

\begin{lemma}\label{groupoid}
Let $\calg$ be a groupoid, and let $M\colon\calg\to\Ab$ be a
functor. Then
\[\lim_\calg M \cong \prod_{\calg_i} \lim_{\calg_i}M \cong
\prod_{\calg_i} M(x_i)^{\aut_{\calg_i}(x_i)},\]
where the product runs over connected components $\calg_i$ of $\calg$, and $x_i$ is an arbitrary object in $\calg_i$. In particular if $\calg$ is
discrete, then
\[\lim_{\calg}M \cong \prod_{\calg_i}
M(x_i).\]
Dually,
\[\colimnoarrow_{\calg}{M}\cong\bigoplus_{\calg_i}M(x_i)_{\aut_{\calg_i}(x_i)},\]
where the subscript means coinvariants. In particular if $\calg$ is discrete, then
\[\colimnoarrow_{\calg}M \cong\bigoplus_{\calg_i}
M(x_i).\]

\end{lemma}
\begin{proof}
The first isomorphism is clear. For the second, observe that each
connected component of a groupoid is equivalent as a category to the
group of automorphisms of any object in it. Thus if $x_i$ is an
object in $\calg_i$, then
\[\lim_{\calg_i}M \cong \lim_{\calb(\aut_{\calg_i}(x_i))}M =
M(x_i)^{\aut_{\calg_i}(x_i)},\] as claimed. The claim for colimits
follows at once by duality. The conclusions for discrete groupoids
are clear.\end{proof}

We are now ready to describe the right Kan Extension of a module
along the inclusion $\iota\colon\call_H\to\call$. To make notation
less cumbersome, we will use the symbol $\widebar{g}$ to denote a left
coset $gH\in \Gamma/H$. Also for any morphism $\varphi$ in $\call$,
let $t_\varphi\defeq \widehat{\Theta}(\varphi)\in\Gamma$.

\newcommand{\rcos}[1]{\widebar{#1}}
\newcommand{\alphag}[1]{\sigma_{#1}}

\begin{lemma}\label{prop:RiMdescriptionR}
Let $M\in\call_H\mod$ be any functor, and fix a section
$\sigma\colon \Gamma/H\to\aut_\call(S)$ of the projection. Let
$\alphag{g}$ denote $\sigma(\rcos{g})$. Then the right Kan extension
of $M$ to $\call$ can be described as follows.
\begin{enumerate}
\item[(a)] For each $P\in\call_H$, there is an isomorphism
\[\Phi_P^\sigma \colon R_\iota (M)(P)\xrightarrow{\cong}
\prod_{\rcos{g}\in \Gamma/H}M(P^{\alphag{g}^{-1}}), \] given by the
formula
\[(\Phi_P^\sigma(\{x_{(Q,\beta)}\}_{(Q,\beta)\in P\da\iota}))_{\rcos{g}} =
x_{(P^{\alphag{g}^{-1}}, \alphag{g}^{-1})}.\] In other words, $\Phi_P^\sigma$
takes a compatible family
\[\{x_{(Q,\beta)}\}_{(Q,\beta)\in P\da\iota} \in R_\iota(M)(P) \defeq \lim_{P\da\iota}
M_\sharp \leq \prod_{(Q,\beta)\in P\downarrow\iota}M(Q) \] to the
element in $\prod_{\rcos{g}\in \Gamma/H}M(P^{\alphag{g}^{-1}})$ whose
$\rcos{g}$-th coordinate is $x_{(P^{\alphag{g}^{-1}}, \alphag{g}^{-1})}$.
\item[(b)] If $\tau \colon \Gamma/H\to\aut_\call(S)$ is another section, then
one has a commutative diagram
\[ \xymatrix{
 && \prod_{\rcos{g}\in \Gamma/H}M(P^{\alphag{g}^{-1}}) \ar[dd]^{\prod M(\tau_g^{-1}
 \circ\alphag{g})}_{\cong}  \\
 R_\iota (M)(P) \ar[urr]^{\cong}_{\Phi_P^\sigma} \ar[drr]_{\cong}^{\Phi_P^{\tau}}& \\
 && \prod_{\rcos{g}\in \Gamma/H}M(P^{\tau_g^{-1}}).
}
\]

\item[(c)]  Let $\varphi\colon P\to Q$ be a morphism in $\call$. Then,
under the identification in part (a), the map
  \[R_\iota (M)(\varphi) \colon R_\iota M(P) \to R_\iota M(Q)\]
sends
   \[(x_{g})_{\rcos{g}\in  \Gamma/H} \in \prod_{\rcos{g}\in  \Gamma/H} M(P^{\alphag{g}^{-1}})\]
to
\[\left\{M(\alphag{g}^{-1}\circ \varphi \circ
\alphag{\widehat{\Theta}(\varphi)^{-1}g})(x_{\widehat{\Theta}(\varphi)^{-1}g})\right\}_{\rcos{g}
\in \Gamma/H} \in \prod_{\rcos{g}\in \Gamma/H}
M(Q^{\alphag{g}^{-1}}).\]
\end{enumerate}
\end{lemma}
\begin{proof}
For each object $P\in\call$, the section $\sigma$ defines an
equivalence of categories $t_P^\sigma\colon \Gamma/H\to \calg_H(P)$
which takes an object $gH$ to $(P^{\sigma_g^{-1}},\sigma_g^{-1})$, and is a
right inverse for $\delta_P\colon \calg_H(P)\to \Gamma/H$ of Lemma
\ref{CH(P)-groupoid}. Let $j_P\colon\calg_H(P)\to P\downarrow\iota$
and $r_P\colon P\downarrow\iota\to \calg_H(P)$ denote the inclusion
and retraction of Lemma \ref{RetractionR} respectively.

The map $\Phi_P^\sigma$ is the map induced by the composite
$j_P\circ t_P^\sigma\colon \Gamma/H\to P\da\iota,$
\[R_\iota(M)(P)\defeq\lim_{P\da\iota}M_\sharp \xrightarrow{j_P^*}
\lim_{\calg_H(P)}M_\sharp\circ j_P \xrightarrow{(t_P^\sigma)^*} \lim_{\Gamma/H}M_\sharp\circ
j_P\circ t_P^\sigma = \prod_{\rcos{g}\in\Gamma/H}M(P^{\sigma_g^{-1}}).\]
Notice that $j_p\circ t_P^\sigma$ takes a coset $gH\in\Gamma/H$ to
the object $(P^{\sigma_g^{-1}},\sigma_g^{-1})\in P\da\iota$. Thus the formula
of part (a) is clear. Furthermore, the first map in the sequence
above is an isomorphism since $\calg_H(P)$ is a left deformation
retract of $P\da\iota$ by Lemma \ref{RetractionR}, and the second is
an isomorphism since $t_P^\sigma$ is an equivalence of groupoids.

For Part (b), notice that the morphisms
$\tau_g^{-1}\circ\sigma_g$ define a natural isomorphism from
$t_P^\sigma$ to $t_P^{\tau}$, and hence a natural isomorphism
$j_P\circ t_P^{\sigma}\to j_P\circ t_P^{\tau}$. Thus
$\Phi_P^\sigma$ and $\Phi_P^{\tau}$ are naturally isomorphic, and
(b) is nothing but a diagramatic realization of this natural
isomorphism.

To prove part (c), notice first that for each $P\in\call$ the
composite functor $\delta_P\circ r_P$ is a left inverse to $j_P\circ
t_P^\sigma$, and since the latter induces the isomorphism
$\Phi_P^\sigma$ of part (a), $\delta_P\circ r_P$ induces
$(\Phi_P^\sigma)^{-1}$. Thus to prove the statement we need to
calculate the effect of $\Phi_Q^\sigma\circ
R_\iota(M)(\varphi)\circ(\Phi_P^\sigma)^{-1}$ on a typical element
$\{x_g\}_{\widebar{g}\in\Gamma/H}\in\prod_{\rcos{g}\in\Gamma/H}M(P^{\sigma_g^{-1}})$.
By direct calculation one has for $(R,\beta)\in P\da\iota$,
$(T,\gamma)\in Q\da\iota$, and $u\in\Gamma$,
\begin{itemize}
  \item $((\Phi_P^\sigma)^{-1}(\{x_g\}_{\widebar{g}\in\Gamma/H}))_{(R,\beta)}
  =
M(\beta\sigma_{\widehat{\Theta}(\beta)^{-1}})(x_{\widehat{\Theta}(\beta)^{-1}})$
  \item $(R_\iota(M)(\varphi)(\{M(\beta\sigma_{\widehat{\Theta}(\beta)^{-1}})(x_{\widehat{\Theta}(\beta)^{-1}})\}_{(R,\beta)}))_{(T,\gamma)}
  =
M(\gamma\varphi\sigma_{\widehat{\Theta}(\gamma\varphi)^{-1}})(x_{\widehat{\Theta})(\gamma\varphi)^{-1}})$
  \item $(\Phi_Q^\sigma(\{M(\gamma\varphi\sigma_{\widehat{\Theta}(\gamma\varphi)^{-1}})
(x_{\widehat{\Theta}(\gamma\varphi)^{-1}})\}_{(T,\gamma)}))_{\widebar{u}}
  =
M(\sigma_{u}^{-1} \varphi\sigma_{\widehat{\Theta}(\sigma_{u}^{-1}\varphi)^{-1}})
(x_{\widehat{\Theta}(\sigma_{u}^{-1}\varphi)^{-1}}),$
\end{itemize}
where $\sigma_{u}^{-1}$ represents the restricted map $Q  \xrightarrow{\sigma_{u}^{-1}} Q^{\sigma_{u}^{-1}}$.
Since $\widehat{\Theta}(\sigma_{u}^{-1}\varphi)^{-1} = \widehat{\Theta}(\varphi)^{-1}u$, we have
\[(\Phi_Q^\sigma\circ
R_\iota(M)(\varphi)\circ(\Phi_P^\sigma)^{-1}(\{x_g\}_{\widebar{g}\in\Gamma/H}))_{\widebar{u}}
= M(\sigma_u^{-1} \circ \varphi \circ \sigma_{\widehat{\Theta}(\varphi)^{-1}u})
(x_{\widehat{\Theta}(\varphi)^{-1}u}),\] as claimed.
\end{proof}

The following is an immediate corollary of  Lemmas \ref{prop:RiMdescriptionR} and \ref{lem:Shapiro}.

\begin{corollary}\label{iota*preserves-projectives} The right Kan extension functor
\[R_\iota \colon \call_H\mod \rightarrow \call\mod  \]
is exact. Thus, the Shapiro map
\[ \Sh_M \colon H^{*}(\call, R_\iota(M)) \rightarrow H^{*}(\call_H, M)  \]
is an isomorphism for every $M \in \call_H \mod$.
\end{corollary}

There is a dual version of both Lemma \ref{prop:RiMdescriptionR}, describing
the left Kan extension $L_\iota M$, and Corollary \ref{iota*preserves-projectives},
showing that $L_\iota$ is exact and hence $\iota^*$ preserves injectives and
the Shapiro lemma holds in homology. We will not state these results explicitly,
but we draw the following corollary.

\begin{corollary} \label{iota*preserves-resolutions}
The restriction functor $\iota^*\colon\call\mod\to\call_H\mod$ preserves both
injective and projective resolutions.
\end{corollary}


\section{Construction of the transfer for subsystems of index prime to $p$}

Throughout this section let $\Gamma=\Gamma_{p'}(\calf)$, fix a
subgroup $H\le\Gamma$, and let $\iota\colon \call_H\to \call$ be
the inclusion. Let $\sigma \colon \Gamma/H\to\aut_\call(S)$ be a
fixed section. As before, for $g\in\Gamma$ we denote by $\rcos{g}$
the left coset $gH\in \Gamma/H$, and let $\sigma_{g}$ denote
$\sigma(\rcos{g})$.

\begin{proposition}\label{Pretransfer}
Let $M \in \call \mod$. For each \mbox{$P \in
\call$}, let
\[\PreTr_P \colon R_{\iota}(\iota^*M)(P) \longrightarrow M(P) \]
be the map that, under the identification of Lemma
\ref{prop:RiMdescriptionR}, is given by
\[(x_g)_{\rcos{g} \in \Gamma/H} \mapsto \sum_{\rcos{g}
\in \Gamma/H} M(\alphag{g})(x_g).\] Then, the maps $\PreTr_P$
are independent of the choice of the section $\sigma$, and assemble
into a natural transformation
\[\PreTr\colon R_{\iota}(\iota^*M) \Rightarrow M \]
\end{proposition}
\begin{proof}
Let $\tau$ be another section. Then by Part (b) of Lemma
\ref{prop:RiMdescriptionR} (omitting $\iota^*$ from the notation),
\[\Phi_P^{\tau} = \left(\prod M(\tau_{\rcos{g}}^{-1}  \circ\sigma_{\rcos{g}})\right)\circ
\Phi_P^\sigma.\] Let $x\in R_\iota(\iota^*M)(P)$, and let
$(x_{g})_{\rcos{g}}$ and $(x'_{g})_{\rcos{g}}$ denote
$\Phi_P^\sigma(x)$ and $\Phi_P^{\tau}(x)$ respectively. Then
\[(x'_{g})_{\rcos{g}} = (M(\tau_{g}^{-1}\circ\sigma_{g})(x_{g}))_{\rcos{g}}.\]
Thus using the section $\tau$ to define $\PreTr$ gives the map
\begin{multline*}
(x'_{g})_{\rcos{g} \in \Gamma/H} \mapsto \sum_{\rcos{g} \in
\Gamma/H} M(\tau_g)(x'_g) 
= \sum_{\rcos{g} \in
\Gamma/H}
M(\tau_g)(M(\tau_{g}^{-1}\circ\sigma_{g})(x_{g})) =
\sum_{\rcos{g} \in \Gamma/H} M(\alphag{g})(x_g),
\end{multline*}
which agrees with the map defined using $\sigma$.

It remains to show that $\PreTr$ is natural, meaning that for  $P,Q \in \call$ and $\varphi \in \mor_\call(P,Q)$, the square
\[
\xymatrix{
  {R_\iota(\iota^*M)(P)} \ar[0,2]^{\PreTr_P} \ar[2,0]_{R_\iota(\iota^*M)(\varphi)} & & M(P) \ar[2,0]^{M(\varphi)}\\ \\
  R_\iota(\iota^*M)(Q) \ar[0,2]^{\PreTr_Q } & & M(Q)}
\]
commutes. Let \mbox{$x \in  R_\iota(\iota^*M)(P)$}, and let
\mbox{$\displaystyle (x_{g})_{\rcos{g}} = \Phi_P^\sigma(x)
\in\prod_{\rcos{g}\in \Gamma/H}M(P^{\alphag{g}^{-1}})$}, as in
Lemma \ref{prop:RiMdescriptionR}.  Then
\[M(\varphi) \circ \PreTr_P (x) = \sum_{\rcos{g} \in \Gamma/H} M(\varphi\circ \alphag{g}) (x_{g}),\]
while, by part (c) of Lemma \ref{prop:RiMdescriptionR},
\begin{align*}
 \PreTr_Q \circ R_\iota(\iota^*M)(\varphi)(x)
 &= \sum_{\rcos{g} \in \Gamma/H} M(\alphag{g})\left(M(\alphag{g}^{-1}\circ
    \varphi\circ\alphag{\widehat{\Theta}(\varphi)^{-1}g})
    (x_{\widehat{\Theta}(\varphi)^{-1}g}) \right) \\
 &= \sum_{\rcos{g} \in \Gamma/H}M(\varphi\circ\alphag{\widehat{\Theta}
    (\varphi)^{-1}g})(x_{\widehat{\Theta}(\varphi)^{-1}g}).
\end{align*}
Thus the two ways of composition in the square coincide. \end{proof}

The natural transformation $\PreTr$ induces a map on higher limits,
which we also denote by $\PreTr$. Composing this with the inverse of
the Shapiro isomorphism we obtain our transfer.

\begin{definition}\label{transfer-defi}
For any subgroup $H\le \Gamma$ and a functor $M\in\call\mod$ the associated \emph{transfer}, is the map
\[ \Tr \colon H^*(\call_H, \iota^*M) \xrightarrow{\Sh_M^{-1}}_{\cong} H^*(\call, R_\iota(\iota^*M)) \xrightarrow{\PreTr} H^*(\call ,M). \]
\end{definition}
Since the transfer is defined as a natural transformation of
coefficient systems it is of course independent of the choice of resolutions used to compute cohomology.

\subsection{Cochain-level Description} \label{sub:CochainLevel}
Just as in the classical group case, the transfer map constructed
here can be described at the cochain level as a sum  of conjugation
maps. To this end we must first make sense of conjugation maps
in our current context.

\begin{definition} Let $M$ and $N$ be coefficient systems on $\call$, and
let $K$ be a subgroup of $\Gamma$. Let $\iota_K\colon \call_K\to \call$ denote the inclusion.  For a natural transformation
  \mbox{$\phi \in \hom_{\call_K\mod}( \iota_K^*N,\iota_K^*M)$},
and
  \mbox{$\alpha \in \aut_{\call}(S)$},
let $\phi^\alpha$ be the map which associates with an object
$P\in\call$ the homomorphism of $R$-modules $\phi_P^\alpha\colon
N(P)\to M(P)$ given by
\[ \phi^\alpha_P := M(\alpha) \circ \phi_{\alpha^{-1}(P)} \circ N(\alpha^{-1}) \in \hom_R(N(P),M(P)). \]
\end{definition}

We refer to this construction as conjugation. Notice that
$\phi^\alpha$ needs not be a natural transformation in general. In the case where $(\phi^\alpha)$ is a natural transformation, so $(\phi^\alpha)^\beta$ is defined, it follows immediately from the definition that $(\phi^\alpha)^\beta = \phi^{(\beta \circ \alpha)}$, the order being reversed since we write
$\alpha$ on the right. The following lemma investigates other basic
properties of $\phi^\alpha$, and in particular provides a condition
under which it is a natural transformation.

\begin{lemma} \label{lem:sigma^aNatural} Let $L, M$ and $N$ be coefficient systems on $\call$,
let $K, H \le\Gamma$, and let \mbox{$\alpha \in
\aut_{\call}(S)$}.
\begin{enumerate}
  \item[(a)] If
\[ \iota_K^* N \xrightarrow{\phi} \iota_K^* M \xrightarrow{\eta}\iota_K^* L \]
is a composable sequence in $\call_K\mod$ then
\[(\eta \circ \phi)^\alpha = \eta^\alpha \circ \phi^\alpha. \]
  \item[(b)]
If $\phi \in \hom_{\call_K\mod}(\iota_K^* N,\iota^*_K M) $,
then for every $\alpha \in \aut_{\call_K}(S)$,
  \[ \phi^\alpha = \phi.\]
  \item[(c)]
If
  \mbox{$\widehat{\Theta}(\alpha)^{-1} \in N_\Gamma(H,K)$}
and
  \mbox{$\phi \in \hom_{\call_K\mod}(\iota_K^* N,\iota^*_K M)$},
then $\phi^\alpha$ is a natural transformation
  \mbox{$\iota^*_H N \to \iota^*_H M$},
and the map which associates $\phi^\alpha$ with $\phi$ is  a homomorphism
 \[ c_\alpha^* \colon  \hom_{\call_K\mod}(\iota_K^* N,\iota^*_K M) \longrightarrow \hom_{\call_H\mod}(\iota_H^* N,\iota^*_H M).\]
\end{enumerate}
\end{lemma}
\begin{proof}
For $P\in\call$ and $\alpha\in\aut_\call(S)$,
\begin{multline*}
  (\eta\circ\phi)^\alpha_P
  \defeq L(\alpha)\circ(\eta\circ\phi)_{\alpha^{-1}(P)}\circ N(\alpha^{-1})=\\
  L(\alpha)\circ\eta_{\alpha^{-1}(P)}\circ M(\alpha^{-1})\circ M(\alpha)\circ \phi_{\alpha^{-1}(P)}\circ N(\alpha^{-1})
  = \eta_P^\alpha\circ\phi_P^\alpha.\end{multline*}
This proves part (a).

Part (b) is immediate, since by assumption
$\phi$ is natural with respect to any morphism in $\call_K$, and in
particular $\alpha$. Therefore, for any object $Q$,
  \[ \phi^\alpha_Q = M(\alpha) \circ \phi_{\alpha^{-1}(Q)} \circ N(\alpha^{-1})
                = M(\alpha) \circ M(\alpha^{-1}) \circ \phi_Q
                = \phi_Q. \]

It remains to prove part (c).
Let \mbox{$\varphi \colon P \to Q$} be a morphism in $\call_H$.
Then \mbox{$\widehat{\Theta}(\varphi) \in H$}, so by assumption on $\alpha$ we have
\[ \widehat{\Theta}(\alpha^{-1} \circ \varphi \circ \alpha) =\widehat{ \Theta}(\alpha)^{-1}
\circ\widehat{\Theta}(\varphi) \circ \widehat{\Theta}(\alpha) \in K, \]
so
\[ \alpha^{-1} \circ \varphi \circ \alpha \in \mor_{\call_K}(\alpha^{-1}(P),\alpha^{-1}(Q)).\]
Thus,
\begin{align*}
  \phi^\alpha_Q \circ N(\varphi)
  &= M(\alpha) \circ \phi_{\alpha^{-1}(Q)} \circ N(\alpha^{-1}) \circ N(\varphi) \\
  &= M(\alpha) \circ \phi_{\alpha^{-1}(Q)} \circ N(\alpha^{-1} \circ \varphi \circ \alpha ) \circ N(\alpha^{-1}) \\
  &= M(\alpha) \circ M(\alpha^{-1} \circ \varphi \circ \alpha ) \circ \phi_{\alpha^{-1}(P)} \circ N(\alpha^{-1}) \\
  &= M(\varphi) \circ M(\alpha ) \circ \phi_{\alpha^{-1}(P)} \circ N(\alpha^{-1}) \\
  &= M(\varphi) \circ \phi^{\alpha}_P.
\end{align*}
This shows that $\phi^\alpha$ is a natural transformation. That the map $c_\alpha^*$ is a
homomorphism follows directly from the definition.
\end{proof}

\begin{remark}\label{consequences of 4.4}
A consequence of part (b) of  Lemma \ref{lem:sigma^aNatural} is that
if $\alpha, \alpha' \in\aut_{\call}(S)$ with
\mbox{$\widehat{\Theta}(\alpha) = \widehat{\Theta}(\alpha')$} then
\mbox{$c_ \alpha ^* = c_{\alpha'}^*$}, since $\alpha ^{-1} \circ
\alpha' \in \aut_{\call_H}(S)$ for any subgroup $H\le\Gamma$.
Therefore we will often write $c_x^*$ for \mbox{$x \in \Gamma$} to
denote the conjugation induced by any \mbox{$\alpha\in
\aut_{\call}(S)$} with \mbox{$\widehat{\Theta}(\alpha) = x$}, and
write $\phi^x$ for $c^*_x(\phi)$. When restricted to
$\hom_{\call_H\mod}(\iota_H^* N,\iota^*_H M)$, the conjugation $c^*_
\alpha $ only depends on the coset $\widehat{\Theta}(\alpha)H$.
Therefore we will also write $c_x^*$ and $\phi^x$ for
\mbox{$\rcos{x} \in  \Gamma/H$}. Finally, we also write $c_x^*$ and
$\phi^x$ for \mbox{$\rcos{x} \in K \backslash \Gamma / H$} when the
maps involved are independent of the choice of representative for
double cosets.
\end{remark}

\begin{corollary}\label{conj-res-indep}
Let $M$ be a coefficient system on $\call$ with injective resolution
$M \to I_\bullet$ and let $K, H\le \Gamma$. For $\alpha \in \aut_{\call}(S)$
such that $\widehat{\Theta}(\alpha)^{-1} \in N_\Gamma(H,K)$, the map induced by conjugation
\[ c_\alpha^* \colon  \hom_{\call_K\mod}(\iota_K^* \R ,\iota^*_K I_\bullet)
\longrightarrow \hom_{\call_H\mod}(\iota_H^* \R ,\iota^*_H I_\bullet) \]
is a map of cochain complexes, and the induced map on cohomology
\[c_\alpha^* \colon H^{*}(\call_K, \iota_K^*M)  \longrightarrow H^{*}(\call_H, \iota_H^*M).\]
is independent of the choice of the injective resolution
$I_\bullet$.
\end{corollary}
\begin{proof}
First recall that $\iota^*_H$ and $\iota^*_K$ preserve injective
resolutions by Corollary \ref{iota*preserves-resolutions}.
Now, denote the differential of $I_\bullet$ by $\partial$. Then, by Lemma
\ref{lem:sigma^aNatural}(b), \mbox{$\partial^\alpha =
\partial$} for any $\alpha\in\aut_\call(S)$. The differentials of the
cochain complexes $\hom_{\call_K\mod}(\iota_K^* \R,\iota^*_K
I_\bullet)$ and $\hom_{\call_H\mod}(\iota_H^* R,\iota^*_H I_\bullet)$ are
both given by composition with $\partial$, and will be denoted
$\delta_K$ and $\delta_H$ respectively. For \mbox{$\phi \in
\hom_{\call_K\mod}(\iota_K^* \R,\iota^*_K I_\bullet)$} we therefore have
\[ c_\alpha \circ \delta_K (\phi) = c_\alpha(\partial \circ \phi) =
\partial^\alpha \circ \phi^\alpha  = \partial \circ \phi^\alpha=
\delta_H\circ c_\alpha (\phi). \] This shows that \mbox{$c_\alpha
\circ \delta_K = \delta_H \circ c_\alpha$}, so $c_\alpha$ is a
cochain map.

The proof of independence on choice of injective resolutions is
routine.
\end{proof}

\begin{lemma} Let $M$ and $N$ be coefficent systems on $\call$, and let
$H\le\Gamma$. If \mbox{$\phi \in
\hom_{\call_H\mod}(\iota_H^*N,\iota_H^*M) $} then
 \[ \sum_{\rcos{g} \in \Gamma/H} \phi^g \in \hom_{\call\mod}(N,M). \]
\end{lemma}
\begin{proof}
We need to show naturality with respect to morphisms in $\call$. So
let \mbox{$\psi \colon P \to Q$} be a morphism in $\call$. Fix a
section $\sigma\colon\Gamma/H \to\aut_{\call}(S)$, and let $\sigma_g$
denote $\sigma(\widebar{g})$, as before. For each $\rcos{g} \in
\Gamma/H$, let $\psi_{g}$ denote the composite
\[ {\sigma_{\widehat{\Theta}(\psi)^{-1}g}^{-1}(P)}
\xrightarrow{\sigma_{\widehat{\Theta}(\psi)^{-1}g}} P \xrightarrow{\psi} Q
\xrightarrow{\sigma_{g}^{-1}} \sigma_{g}^{-1}(Q),   \]
where the appropriate
restrictions on $\sigma_{g}$ and
$\sigma_{\widehat{\Theta}(\psi)^{-1}g}^{-1}$ are understood. Observe that
$\widehat{\Theta}(\psi_{g}) \in H$, so $\psi_{g} \in
\mor_{\call_H}(\sigma_{\widehat{\Theta}(\psi)^{-1}g}^{-1}(P),\sigma_{g}^{-1}(Q))$.

 Now,
\begin{align*}
   \sum_{\rcos{g} \in \Gamma/H} \phi^g_Q \circ N(\psi)
   &= \sum_{\rcos{g} \in \Gamma/H} M(\sigma_{g}) \circ
   \phi_{\sigma_{g}^{-1}(Q)} \circ N(\sigma_{g}^{-1}) \circ N(\psi)  \\
   &= \sum_{\rcos{g} \in \Gamma/H} M(\sigma_{g}) \circ
   \phi_{\sigma_{g}^{-1}(Q)} \circ N(\psi_{g}) \circ N(\sigma_{\widehat{\Theta}(\psi)^{-1}g}^{-1}) \\
   &= \sum_{\rcos{g} \in \Gamma/H} M(\sigma_{g}) \circ M(\psi_{g}) \circ
   \phi_{\sigma_{\widehat{\Theta}(\psi)^{-1}g}^{-1}(P)} \circ N(\sigma_{\widehat{\Theta}(\psi)^{-1}g}^{-1}) \\
   &= \sum_{\rcos{g} \in \Gamma/H} M(\psi) \circ
   M(\sigma_{\widehat{\Theta}(\psi)^{-1}g}) \circ
   \phi_{\sigma_{\widehat{\Theta}(\psi)^{-1}g}^{-1}(P)} \circ
   N(\sigma_{\widehat{\Theta}(\psi)^{-1}g}^{-1}) \\
   &= M(\psi) \sum_{\rcos{g} \in \Gamma/H} \phi^{\widehat{\Theta}(\psi)^{-1}g}_P. \\
   &= M(\psi) \sum_{\rcos{g} \in \Gamma/H} \phi^{g}_P.
\end{align*}
The first and the fifth equalities follows from the definition of
$\phi^g$ (see Remark \ref{consequences of 4.4}), the second and the fourth from the definition $\psi_g$,
the third from naturality of $\phi$, and the sixth is clear. This
shows naturality and proves the claim.
\end{proof}

\begin{definition}\label{TrH}
For a coefficient system $M$ on $\call$, a subgroup $H\le\Gamma$
and an injective resolution $M \to I_\bullet$, the associated
\emph{cochain-level transfer} is the cochain map
\[ \Tr_H \colon \hom_{\call_H}(\iota_H^* \R, \iota_H^*I_\bullet) \longrightarrow \hom_{\call}(\R,I_\bullet) \]
given by
\[ \Tr_H(\phi) = \sum_{\rcos{g} \in \Gamma/H} \phi^{g}.  \]
\end{definition}

Since $Tr_H$ is a sum of conjugations, it follows by Corollary \ref{conj-res-indep} that it  is indeed a cochain map, and that the induced map on cohomology is independent of the
choice of injective resolution.

\begin{proposition} For a coefficient system $M$ on $\call$ and a subgroup $H$ of $\Gamma$, the map
\[H^{*}(\call_H, \iota^*M) \longrightarrow H^{*}(\call, M) \]
induced by the cochain-level transfer $Tr_H$ of Definition \ref{TrH} coincides with the
transfer associated to the same data, as defined in
\ref{transfer-defi}.
\end{proposition}
\begin{proof}
Let $\iota \colon \call_H\to\call$ denote the inclusion.
If $M \to I_\bullet$ is an injective resolution, then,
by Corollary \ref{iota*preserves-resolutions}, $\iota^* M \to \iota^* I_\bullet$
is an injective resolution.
Thus the Shapiro isomorphism $Sh_M$ is given on the
cochain-level by the
$(\iota^*,R_{\iota})$-adjunction isomorphism
\[ \rho \colon \hom_{\call\mod}(\R,R_\iota(\iota^*I_\bullet))   \xrightarrow{\cong}\hom_{\call_H\mod}(\iota^* \R ,\iota^* I_\bullet ).\]
(cf.~Subsection \ref{sub:KanShapiro}). Hence its inverse $Sh_M^{-1}$  is induced by $\rho^{-1}$,
whose value on \mbox{$\phi \in
\hom_{\call_H\mod}(\iota^*\R,\iota^* I_\bullet ) $} is the
natural transformation which takes an object $Q\in\call$ to the morphism $\rho^{-1}(\phi)_Q$ given by the composite
\[ \R = \R(Q)  
 \xrightarrow{\ \prod\limits_{\rcos{g}\in \Gamma/H}  R(\alphag{g}^{-1}) }
 \prod_{\rcos{g}\in \Gamma/H} \R(Q^{\alphag{g}^{-1}}) \xrightarrow{\prod\limits_{\rcos{g}\in \Gamma/H}
 \phi_{\alphag{g}^{-1}(Q)} } \prod_{\rcos{g}\in \Gamma/H} I_\bullet(Q^{\alphag{g}^{-1}}) \cong  R_\iota(\iota^* I_\bullet)(Q) . \]
Therefore the composite
\[ \hom_{\call_H\mod}(\iota^* \R ,\iota^* I_\bullet  ) \xrightarrow{\rho^{-1}_M}_{\cong} \hom_{\call\mod}(\R ,R_\iota(\iota^* I_\bullet))
\xrightarrow{\PreTr} \hom_{\call}(\R ,I_\bullet), \]
which induces the transfer, is given by
\[ \phi \mapsto \sum_{\rcos{g}\in \Gamma/H} I_\bullet(\alphag{g}) \circ \phi_{\alphag{g}^{-1}(Q)} \circ \R(\alphag{g}^{-1})
              = \sum_{\rcos{g}\in \Gamma/H} \phi^{\alphag{g}} = \Tr_H(\phi). \]
\end{proof}

\subsection{Geometric Interpretation}\label{sub:geometric}
We end this section by verifying that for a locally
constant system of coefficients $M\in\call\mod$, and a subgroup
$\call_H\le\call$ of index prime to $p$, the transfer defined in
this section coincides with the map defined in Section
\ref{Loc-Const-coeff}.

Let $\iota\colon \call_H\to \call$ denote the inclusion. Let $\Gamma
= \pi_1(|\call|, S)$ and $\Gamma' = \pi_1(|\call_H|, S)$. Then, by
Theorem \ref{Classification-summary}(v), the map induced by $\iota$
on nerves is a covering space, up to homotopy, with homotopy fibre
$\Gamma/\Gamma' \cong \Gamma_{p'}(\calf)/H$.  Construct $\pi\colon
\call_{\Gamma'}\to \call$ as in Section
\ref{cov-sp-loc-cosnt-coeff}. Then, by Lemma \ref{categorical
covering}(i), there is a functor $\widehat{\iota}\colon\call_H\to
\call_{\Gamma'}$, such that $\pi\circ\widehat{\iota} = \iota$, and
by Part (ii) of the same lemma, for any locally constant
$M\in\call\mod$, there is a natural isomorphism induced by
$\widehat{\iota}$, \[\widehat{\iota}_*\colon
R_\pi(\pi^*M)\xrightarrow{\cong} R_\iota(\iota^*M).\]

Choose a section $\sigma\colon \Gamma/H\to\aut_\call(S)$ of the
projection, and construct $\Phi^\sigma$ as in
Lemma \ref{prop:RiMdescriptionR}. Then for any $M\in\call\mod$
and $P\in \call$, we claim that the diagram
\[\xymatrix{
R_\pi(\pi^*M)(P) \ar[rr]^{Pr}\ar[d]^{\widehat{\iota}_*} && \prod_{\Gamma/\Gamma'}M(P) \ar[rr]^{\Sigma}\ar[d]^{\prod M(\sigma_g)^{-1}}&& M(P)\ar@{=}[d]\\
R_\iota(\iota^*M)(P) \ar[rr]^{\Phi^\sigma_P}&&
\prod_{\widebar{g}\in\Gamma/\Gamma'}M(P^{\sigma_g^{-1}}) 
\ar[rr]^{\Sigma^{\sigma}} && M(P)
}\]
commutes. Here the map $Pr$ projects a compatible family of elements
in the source onto the coordinates of the respective initial
objects, i.e., to the coordinates of the objects of the form
$((P,g\Gamma'),1_P)$. The map $\Sigma$ is the sum of coordinates,
where $\Sigma^\sigma$ is the twisting of the sum of coordinates map
by the $M(\sigma_g)$, as in the construction of $\PreTr_P$ in
Proposition \ref{Pretransfer}. The right square clearly commutes,
while the verification that the left square commutes only involves
checking all the relevant definitions.

Since  the map $T_P$  of Proposition \ref{RKE along hty covering} is the top  composition, and the map $\PreTr_P$ is the bottom composition, it follows that one has a commutative diagram in cohomology
\[\xymatrix{
H^*(\call_{\Gamma'}, \pi^*M) \ar[rr]^{\cong} \ar[d]^{\widehat{\iota}^*}_\cong && H^*(\call, R_\pi(\pi^*M))  \ar[rr]^{T_*} \ar[d]^{\widehat{\iota}_*}_\cong&& H^*(\call, M) \ar@{=}[d]\\
H^*(\call_H, \iota^*M) \ar[rr]^{\cong} && H^*(\call, R_\iota(\iota^*M)) \ar[rr]^{\PreTr_*} && H^*(\call, M).
}\]
Here the top composition is the transfer associated to the covering
 $|\call_H|\to
|\call|$ by Proposition \ref{plf-subgrps-loc-const-coeff}, while the
bottom composition is our algebraic transfer.


\section{Standard Consequences}

In this final section we show that the transfer constructed here for subsystems of index prime to $p$
satisfies all the standard properties one expects to have in ordinary
group cohomology.  Throughout this section
$\Gamma=\Gamma_{p'}(\calf)$, and we let
$\widehat{\Theta}\colon\call\to \calb(\Gamma)$ denote the
projection functor, as before.

\subsection{Transfer among Subgroups} \label{sub:HKTransfer}
We have defined a transfer associated to the inclusion $\call_H
\subseteq \call$ for a subgroup $H \leq \Gamma$. Just like in the
group case, this construction can be generalized to obtain a
transfer associated to the inclusion $\call_H \subseteq \call_K$ for
subgroups $H \leq K \leq \Gamma$. Indeed, one can simply take
$\call_K$ in place of $\call$ in the construction of the transfer.
The transfer associated to the inclusion $\iota^K_H\colon\call_H
\subseteq \call_K$ can then be described on the cochain level as
follows. Let $M$ be a coefficient system on $\call_K$, and let $M \to I_\bullet$ be an injective resolution of $M$. A cochain-level transfer
  \[ \Tr_H^K \colon \hom_{\call_H}(\iota_H^* \R,(\iota_H^K)^*I_\bullet)
  \longrightarrow \hom_{\call_K}(\iota_K^* \R, I_\bullet) \]
is given by
\[ \Tr_H^K(\phi) = \sum_{\rcos{x} \in K/H} \phi^x,  \]
as in Definition \ref{TrH}. This formula makes it clear that $\Tr_H
= \Tr_H^K \circ \Tr_K$ for subgroups $H \leq K \leq \Gamma$.

\subsection{Normalization}
One of the most basic properties of the standard cohomology transfer is that
the restriction to the cohomology of a subgroup followed by the transfer is
given by multiplication by the index. This is exactly the case in our more general context.

\begin{proposition} \label{prop:Norm} Let $\SFL$ be a $p$-local finite group, and let
$H$ be a subgroup of $\Gamma = \Gamma_{p'}(\calf)$. Let $M$ be a system of coefficients
for $\SFL$. Then the composite
\[H^*(\call, M) \xrightarrow{\Res} H^*(\call_H, \iota^*M)\xrightarrow{Tr} H^*(\call, M)\]
is given by multiplication by the index $|\Gamma\colon H|$.
\end{proposition}
\begin{proof} Consider the diagram
\[
\xymatrix{
  {H^*(\call, M)} \ar[0,3]^{\Res} \ar[2,3]^{(\delta_M)_*} &&& {H^*(\call_H, \iota^*M)}
  \ar[2,0]^{\Sh_M^{-1}}_{\cong}  \ar[0,3]^{Tr_H}       & &
  & {H^*(\call, M)} \\ \\
                                                                           & &
  & {H^*(\call, R_\iota(\iota^*M))} \ar[-2,3]^{\PreTr},}
\]
where the right triangle commutes by definition of the transfer, and
where the morphism $\delta_M \colon M \to R_\iota(\iota^*M)$ is the unit of the $(\iota^*,R_\iota)$-adjunction evaluated at the object $M$. Notice that under the identification of Lemma \ref{prop:RiMdescriptionR}(a), $\delta_M(x)= \{M(\sigma_g)^{-1}(x)\}_{\rcos{g}\in\Gamma/H}$.

The  left triangle is commutative by construction of the Shapiro
isomorphism (see Subsection \ref{sub:KanShapiro}). To see this, notice that $Sh_M^{-1}$ is induced by the adjunction isomorphism
\[\rho\colon \hom_{\call_H\mod}(\R_{\call_H}, \iota^*I_\bullet) =  \hom_{\call_H\mod}(\iota^*\R_{\call}, \iota^*I_\bullet) \to \hom_{\call\mod}(\R_\call, R_\iota(\iota^*I_\bullet)),\]
where $I_\bullet$ is an injective resolution of $M$ in $\call\mod$. The adjunction, in turn, applied to a natural transformation $\psi\colon\iota^*\R_\call\to\iota^*I_\bullet$ is the composite
\[\R_\call\xrightarrow{\delta_{\R_\call}} R_\iota(\iota^* \R_\call)\xrightarrow{R_\iota\psi} R_\iota(\iota^* I_\bullet).\] By naturality of $\delta$, for any natural transformation $\varphi\colon \R_\call\to I_\bullet$, \[\delta_{I_\bullet}\circ\varphi = R_\iota(\iota^*\varphi)\circ\delta_{\R_\call}.\]
This proves commutativity of the left triangle, and therefore it suffices to show that the composite $\PreTr \circ
(\delta_M)_*$ is given by multiplication by $|\Gamma:H|$. This map
is induced by a natural transformation of coefficients
\[ M \xrightarrow{\delta_M} R_\iota(\iota^*M) \xrightarrow{\PreTr} M,\]
which is given element-wise by
\[
   x \mapsto \{M(\sigma_g)^{-1}(x)\}_{\rcos{g} \in \Gamma/H} \mapsto \sum_{\rcos{g} \in \Gamma/H}
   M(\sigma_g)(M(\sigma_g^{-1})(x)) = |\Gamma:H| \cdot x
\]
for $P \in \call$ and $x \in M(P).$
\end{proof}

\subsection{Double Coset Formula}
We now show that the transfer map constructed here satisfies a \emph{double coset formula} which is essentially identical to that which holds for ordinary group cohomology.

\begin{proposition}\label{double-cosets}
Let $\SFL$ be a $p$-local finite group, and let $H$ and  $K$ be subgroups of  $\Gamma = \Gamma_{p'}(\calf)$. Then the composite
\[ H^*(\call_K, \iota_K^*M) \xrightarrow{Tr_K}  H^*(\call, M)\xrightarrow{\Res_H} H^*(\call_H, \iota_H^*M) \]
is given by
\[\Res_H \circ \Tr_K =   \sum_{x \in H \backslash\Gamma/K} \Tr_{H\cap xKx^{-1}}^H \circ c_x^* \circ \Res_{{x^{-1}Hx\cap K}}^{K} \, .\]
\end{proposition}
\begin{proof}
Using the cochain-level description of the transfer in Section
\ref{sub:CochainLevel}, one can easily adapt most textbook proofs of
the double coset formula for group cohomology to prove this
proposition. The argument presented here follows the lines of  \cite[Thm 6.2]{AM}, to which the reader is referred for full detail.

Let \mbox{$\Gamma = \coprod_i H x_i K$} be a double coset
decomposition of $\Gamma$ with respect to $K$ and $H$. For each $i$,
put
\[V_i \defeq x_i^{-1}Hx_i\cap K \quad \mathrm{and}\quad W_i \defeq H\cap x_iKx_i^{-1} ,\] and let \mbox{$H = \coprod_j
z_{j i}W_i$} be a left coset decomposition of $H$ with respect to
$W_i$. Then one can rewrite the right $K$-coset decomposition of $\Gamma$ as
\[\Gamma = \coprod_{i,j} z_{j i} x_i K. \]
Let \mbox{$M\to I_\bullet$} be an injective resolution of $M$ in $\call\mod$. Since $x_i^{-1} \in N_\Gamma(W_i,V_i)$ we have induced chain maps
\[c_{x_i}^* \colon \hom_{\call_{V_i}\mod}(\iota_{V_i}^*\R_\call,\iota_{V_i}^*I_\bullet) \longrightarrow \hom_{\call_{W_i}\mod}(\iota_{W_i}^*\R_\call,\iota_{W_i}^*I_\bullet),\]
as in Corollary \ref{conj-res-indep}.
Let \mbox{$\phi \in
\hom_{\call_K\mod}(\iota_K^*\R_\call,\iota_K^*I_\bullet )$}. Then
\begin{align*}
  \Res_H \circ \Tr_K (\phi) &= \Res_H \left( \sum_{\rcos{g} \in  \Gamma/K} \phi^{g}  \right)  \\
                              &=  \Res_H \left( \sum_i \left(\sum_j \phi^{(z_{j i} x_i)}\right) \right) ,
\end{align*}
and
\begin{align*}
     \sum_i \Tr_{W_i}^H \circ c_{x_i}^* \circ \Res_{V_i}^K (\phi)
  &= \sum_i \left(\sum_j \left(c_{x_i}^*(\Res_{V_i}^K(\phi))\right)^{z_{i j}} \right) \\
  &=  \sum_i \left(\sum_j \left(\Res_{V_i}^K(\phi)\right)^{(z_{j i} x_i)}\right) .
\end{align*}

As we can ignore restrictions when determining the effect of these natural transformations, this calculation shows that the double coset formula holds already at the chain-level. Since the maps involved are maps of chain complexes, the result follows.
\end{proof}

\subsection{A Stable Elements Theorem}
In standard group cohomology, the availability of a transfer and a double coset formula allows one to prove that the cohomology of a finite group with a $p$-local module of coefficients is given by the so called "stable elements" in the cohomology of a Sylow $p$-subgroup with the restricted module of coefficients. Using a ``transfer-like'' map, a similar theorem was proved in $p$-local finite group cohomology with $p$-local \emph{constant} coefficients in \cite[Section 5]{BLO2}), and generalized to any stably representable non-equivariant cohomology theory in \cite{KR:ClSpec} (see also \cite{CM}). We now show that the existence of a transfer and a double coset formula in our context implies a stable elements theorem in $p$-local finite group theory for arbitrary $p$-local modules of coefficients, and subgroups of index prime to $p$.

Fix a $p$-local finite group $\SFL$, and let $\Gamma = \Gamma_{p'}(\calf)$.

\begin{definition}
Let $H\le K\le\Gamma$,  and let $M\in\call\mod$ be a system of coefficients. An element
\mbox{$x \in H^*(\call_H,\iota_H^*(M))$} is
\emph{$K$-stable} if for every subgroup \mbox{$U \leq H$} and every $g \in K$ with $g^{-1} \in N_K(U,H)$,
\[ c_g^* \circ \Res_{gUg^{-1}}^H (x) = \Res_U^H(x) \in  H^{*}(\call_U,\iota_U^*(M)).\]
\end{definition}

\begin{lemma} Let $H\le K\le\Gamma$,  and let $M$ be a system of coefficients on $\call$.  The
$K$-stable elements in $H^{*}(\call_H,\iota_H^*(M))$
form a submodule.
Furthermore, if $A$ is a system of ring coefficients, then the
submodule of $K$-stable elements in
$H^{*}(\call_H,\iota_H^*(A))$ is a subring.
\end{lemma}
\begin{proof}
This follows at once from the definition of stable elements, and the fact that the maps induced by restriction and conjugation are module maps or, in the case of ring coefficients, ring maps.
\end{proof}

\begin{lemma}\label{lem:Kstable} Let $H\le K\le\Gamma$,  and let $M\in\call\mod$ be a system of coefficients. If \mbox{$x
\in H^{*}(\call_K,\iota_K^*(M))$} then $\Res_H^K(x)$
is $K$-stable.
\end{lemma}
\begin{proof}
This follows from part (b) of Lemma \ref{lem:sigma^aNatural}.
\end{proof}
We are now ready to state the Stable Elements Theorem for $p$-local finite groups.
\begin{theorem}\label{Stable-Elts} Let $\Gamma=\Gamma_{p'}(\calf)$, let $H\le K\le\Gamma$, and let $M\in\call\mod$ be a $p$-local system of
coefficients. Then the restriction homomorphism
\[\Res_H^K \colon H^{*}(\call_K,\iota_K^*(M)) \longrightarrow H^{*}(\call_H,\iota_H^*(M)) \]
is a split injection whose image is the submodule of $K$-stable
elements.
\end{theorem}
\begin{proof}
By Proposition \ref{prop:Norm}, the composite \mbox{$\Tr_H^K \circ \Res_H^K$} acts on $H^{*}(\call_K,\iota_K^*(M))$ as multiplication by the index $|K\colon H|$. Since this index is prime to $p$, the map \mbox{$t \defeq |K \colon H|^{-1} \cdot\Tr_H^K$} is a left inverse for $\Res_H^K$, proving split injectivity. By Lemma \ref{lem:Kstable} the image of $\Res_H^K$ lies in the submodule of stable elements. Conversely, if \mbox{$a \in H^{*}(\call_H,\iota_H^*(M))$} is $K$-stable, then the double coset formula gives us
\begin{align*}
  \Res_H^K \circ \Tr_H^K (a)
  &=  \sum_{x \in H \backslash K/H} \Tr_{H \cap xHx^{-1}}^H \circ c_x^* \circ \Res_{x^{-1}Hx \cap H}^H (a) \\
  &=  \sum_{x \in H \backslash K/H} \Tr_{H \cap xHx^{-1}}^H \circ \Res_{H \cap xHx^{-1}}^H (a) \\
  &= \sum_{x \in H \backslash K/H} |H \colon H \cap xHx^{-1}| \, a \\
  &= |K \colon H| \, a,
\end{align*}
where the equality
\[ \sum_{x \in H \backslash K/H} |H \colon H \cap x^{-1}Hx| = |K \colon H| \]
is obtained as in the first step of the proof of the double coset
formula. It follows that
\[ x = \Res_H^K \left( \Tr_H^K (|K \colon H|^{-1} \, x)\right) \]
is in the image of $\Res_H^K$.
\end{proof}

\subsection{Frobenius Reciprocity}
Here we show that the transfer and restriction maps for $p$-local finite groups satisfy the standard Frobenius reciprocity formula.

\begin{proposition} \label{prop:Frob} Let $\SFL$ be a $p$-local finite group, and let $\Gamma$ be $\Gamma_p(\calf)$ or $\Gamma_{p'}(\calf)$. Let $H \leq K\le \Gamma$,
and let $A$ be a system of ring coefficients on $\call$. For $x \in H^*(\call_K,\iota_K^* A)$
and $ y \in H^*(\call_H, \iota_H^* A)$ we have
\[ \Tr_H^K(\Res_H^K(x) \, y) = x \, \Tr_H^K(y).  \]
\end{proposition}
\begin{proof}
We prove this at the cochain level. Let $P_\bullet \to R$ be a projective resolution
of the constant functor on $\call$, and let $\phi \in
\hom_{\call_K\mod}(\iota_K^*P_\bullet,\iota_K^*A ),$ (representing $x$) and \mbox{$\psi \in
\hom_{\call_H\mod}(\iota_H^*P_\bullet,\iota_H^*A )$} (representing $y$). Then
\begin{align*}
  \Tr_H^K ( \Res_H^K(\phi) \, \psi )
      &= \sum_{\rcos{g} \in K /H} (\Res_H^K(\phi) \, \psi)^g 
      = \sum_{\rcos{g} \in K/H} \Res_H^K(\phi)^g \, \psi^g \\
      &= \sum_{\rcos{g} \in K/H} \Res_H^K(\phi) \, \psi^g 
      = \phi \sum_{ \rcos{g} \in K/H} \, \psi^g \\
      &= \phi \, \Tr_H^K(\psi),
\end{align*}
where we have used the equality $\Res_H^K(\phi)^g = \Res_H^K(\phi^g) = \Res_H^K(\phi)$ for $g \in K$ obtained in Lemma \ref{lem:sigma^aNatural} (b).
The result follows by passing to cohomology.
\end{proof}

\end{document}